\documentclass{amsart}

\usepackage{epic,eepic,float,fullpage}
\usepackage{amsmath,amsfonts,amssymb,amsthm}
\usepackage{graphicx}
\usepackage{mathtools}
\usepackage[colorinlistoftodos]{todonotes}
\usepackage{cancel}
\usepackage{stmaryrd}
\usepackage{url}
\usepackage{bbm}

\usepackage{hyperref}

\usepackage{etoolbox}
\patchcmd{\section}{\scshape}{\bfseries}{}{}
\makeatletter
\renewcommand{\@secnumfont}{\bfseries}
\makeatother

\DeclarePairedDelimiter\abs{\lvert}{\rvert}%
\DeclarePairedDelimiter\norm{\lVert}{\rVert}%

\makeatletter
\let\oldabs\abs
\def\abs{\@ifstar{\oldabs}{\oldabs*}}
\let\oldnorm\norm
\def\norm{\@ifstar{\oldnorm}{\oldnorm*}}
\makeatother

\makeatletter
\newtheorem*{alt@theorem}{\alt@title}
\newcommand{\newalttheorem}[2]{%
\newenvironment{alt#1}[1]{%
 \def\alt@title{#2 \ref{##1}'}%
 \begin{alt@theorem}}%
 {\end{alt@theorem}}}
\makeatother

\DeclareMathOperator\sgn{sgn}

\newtheorem{prop}{Proposition}[section]
\newalttheorem{prop}{Proposition}
\newtheorem{con}[prop]{Conjecture}
\newtheorem{theo}[prop]{Theorem}
\newtheorem*{theo*}{Theorem}
\newtheorem{lem}[prop]{Lemma}
\newtheorem{cor}[prop]{Corollary}
\newtheorem*{quest*}{Question}
\theoremstyle{definition}

\newtheorem{rem}[prop]{Ramark}

\numberwithin{equation}{section}

\def \be{\begin{equation*}}
\def \ee{\end{equation*}}

\def \ran{\right\rangle}
\def \lan{\left\langle}

\def \D{\mathcal D}

\def \Q{\mathcal Q}
\def \R{\mathbb R}

\def \({\left (}
\def \){\right )}
\def \im {\mathrm{im}}

\renewcommand\P{\mathbb P}
\newcommand\E{\mathbb E}

\newcommand{\1}{\mathbbm 1}

\newcommand \M{\mathcal M_{z,z';\theta}}

\newcommand\calD{\mathcal D}
\newcommand\calE{\mathcal E}
\newcommand \calQ{\mathcal Q}
\newcommand \calP{\mathcal P}

\title{Dirichlet forms of diffusion processes on Thoma simplex}
\author{Sergei Korotkikh}

\begin{document}

\maketitle

\begin{abstract} We study a prominent two-parametric family of diffusion processes $X_{z,z'}$ on an infinite-dimensional Thoma simplex. The family was constructed by Borodin and Olshanski in 2007 and it closely resembles Ethier-Kurtz's infinitely-many-neutral-allels diffusion model on Kingman simplex (1981) and Petrov's extension of Ethier-Kurtz's model (2007). The processes $X_{z,z'}$ have unique symmetrizing measures, namely, the boundary $z$-measures, which play the role of Poisson-Dirichlet measures in our context.

We establish the following behavior of diffusions $X_{z,z'}$: immediately after the initial moment they jump into a dense face of Thoma simplex and then always stay there. In other words, the face acts as a natural state space for the diffusions, while other points of the simplex act like an entrance boundary for our process. As a key intermediate step we study the Dirichlet forms of the diffusions $X_{z,z}$ and find a new description for them. 
\end{abstract}

\tableofcontents

\section{Introduction}

\subsection{Background}

In \cite{EK81} Ethier and Kurtz introduced infinitely-many-neutral-allels model motivated by approximations of models from population genetics. The model itself is a one-parameter family of diffusion processes on the Kingman simplex
\be
\overline{\nabla}_\infty=\{x=(x_1, x_2, x_3, \dots)| x_1\geq x_2\geq x_3\geq\dots\geq0, \sum_{i\geq 1} x_i\leq1\},
\ee
which are defined by a Feller pre-generator acting on an algebra generated by power sums
$$
q_k=\sum_{i\geq 1} x_i^{k+1}, \qquad k\in\mathbb Z_{\geq 1}.
$$ 
Later in \cite{Pe07} Petrov constructed a more general family of diffusion processes $X^{Pet}_{a,\tau}$ on $\overline{\nabla}_\infty$ depending on a pair of real parameters $(a,\tau)$. This generalization is defined by the pre-generator 
\begin{equation}
\label{Pet-gen}
A^{Pet}_{a,\tau}=\sum_{i\geq 1}x_i\frac{\partial^2}{\partial x_i^2}-\sum_{i,j\geq 1}x_ix_j\frac{\partial^2}{\partial x_i\partial x_j}-\sum_{i\geq 1}(\tau x_i+a)\frac{\partial}{\partial x_i},
\end{equation}
which again acts on the subalgebra of functions on $\overline{\nabla}_\infty$ generated by $q_k$. Ethier-Kurtz's model is included in the family $X^{Pet}_{a,\tau}$ as an $a=0$ case. One of remarkable properties of Petrov's model is its association with the celebrated two-parameter family of Poisson-Dirichlet distributions $PD(a, \tau)$ \cite{PY97}, which act as unique symmetrizing measures for diffusions $X^{Pet}_{a, \tau}$. Due to this connection and in attempt to generalize properties of the infinitely-many-neutral-allels model, Petrov's model has been actively studied since its introduction, see \cite{RW09}, \cite{FS09}, \cite{FSWX11}, \cite{Eth14}, \cite{CDERS16}, \cite{FPRW21}. 

Here is one intriguing result about this model. In \cite{Eth14} Ethier showed that the face $\nabla_\infty\subset\overline{\nabla}_\infty$ consisting of points with $\sum_{i\geq 1}x_i=1$ is the natural state space of the processes $X^{Pet}_{\tau,a}$. More precisely, for any starting point $x\in \overline{\nabla}_\infty$, even outside of the face $\nabla_\infty$, we have
\begin{equation}
\label{petface}
\P_x(X^{Pet}_{\tau,a}(t)\in \nabla_\infty\quad \text{for\ all\ } t>0)=1.
\end{equation}
Ethier's elegant proof of this result relied on two ingredients. The first ingredient was the continuity of the transition function of the process with respect to $PD(a, \tau)$, which was proved in \cite{Eth92}, \cite{FSWX11}. The second ingredient was a weaker version of \eqref{petface} where the process starts from its invariant distribution instead of an arbitrary point $x$. To prove this weaker property one needed to study the Dirichlet forms associated to $X^{Pet}_{\tau,a}$, which was first done in \cite{Sch91} for Ethier-Kurtz's model and later in \cite{FS09} for Petrov's diffusion.

The goal behind this work is to translate Ethier's result from \cite{Eth14} to a related family of processes $X_{z,z'}$ introduced by Borodin and Olshanski in \cite{BO07}. This family depends on a couple of complex parameters $(z, z')$ satisfying certain conditions. The state space of these processes is \emph{Thoma simplex}
\be
\Omega=\{(\alpha;\beta)=(\alpha_1, \alpha_2, \dots, \beta_1, \beta_2, \dots)\mid \alpha_1\geq\alpha_2\geq\dots 0,\ \beta_1\geq\beta_2\geq\dots\geq 0,\ \sum_{i\geq 1}\alpha_i+\sum_{j\geq 1}\beta_j\leq 1\}.
\ee
The process is originally constructed as a limit of up-down Markov chains on integer partitions and it can be defined by a pre-generator acting on a certain subspace of functions. The connection with Ethier-Kurtz's and Petrov's model comes from an even more general three-parametric family of diffusions $X_{z,z',\theta}$ on $\Omega$, which was introduced in \cite{Olsh09} and which depends on an additional parameter $\theta>0$. The two-parametric family $X_{z,z'}$ corresponds to the case $\theta=1$ while Ethier-Kurtz's model can be obtained as a $\theta\to 0$ degeneration.

Similarly to the Ethier-Kurtz and Petrov's models, there is a family of measures $\mathcal M_{z,z'}$ on $\Omega$ associated with the diffusions $X_{z,z'}$. These measures are called \emph{$z$-measures} and they originate from harmonic analysis of the infinite symmetric group \cite{KOV93}, \cite{KOV03}. While $z$-measures share some similarities with Poisson-Dirichlet measures, the former have proved to be more challenging to study, partially due to a lack of convenient descriptions of these measures. Nevertheless there are many results regarding these measures, see \cite{BO98}, \cite{BO99}, \cite{Ok00} \cite{BOS05}, \cite{Olsh18}.

Due to the existence of symmetrizing measures $\mathcal M_{z,z'}$, we are able to study an analogue of \eqref{petface} for the process $X_{z,z'}$ using the same two ingredients as in \cite{Eth14}. The first ingredient, the transition density of $X_{z,z'}$ with respect to $\mathcal M_{z,z'}$, was already covered in our previous paper \cite{K18}. The study of the other ingredient, namely the Dirichlet form of $X_{z,z'}$, constitutes the majority of this work. 

\subsection{Main results} Let us be more concrete about the studied processes. To make our results more general we use the three-parametric family of processes $X_{z,z';\theta}$ from \cite{Olsh09} and then we set $\theta=1$ when necessary, obtaining results for the two-parametric family $X_{z,z'}$. For $k\geq 1$ define the moment coordinates on $\Omega$ by 
$$
q_k=\sum_{i\geq 1}\alpha_i^{k+1}+(-\theta)^k\sum_{j\geq 1}\beta_j^{k+1}.
$$
Then $X_{z,z';\theta}$ can be defined as a process with Feller pre-generator
\begin{multline}\label{Agen-intro}
A_{z,z';\theta}=\sum_{k,l\geq 1}(k+1)(l+1)(q_{k+l}-q_kq_l)\frac{\partial^2}{\partial q_k\partial q_l}\\
+\sum_{k\geq 1}(k+1)\left(((1-\theta)k+z+z')q_{k-1}-(k+\theta^{-1}zz')q_k\right)\frac{\partial}{\partial q_k}
+\theta\sum_{k,l\geq 0}(k+l+3)q_kq_l\frac{\partial}{\partial q_{k+l+2}}
\end{multline}
acting on $\mathbb R[q_1, q_2, \dots]$, where we set $q_0=1$.

The process $X_{z,z';\theta}$ has a unique symmetrizing probability measure $\M$. So, we can define the \emph{Dirichlet form} of this process, which is the symmetric form $\mathcal E_{z,z';\theta}(f,g)=-\langle A_{z,z';\theta}f, g\rangle_{L^2(\Omega,\M)}$. This form plays a key role in our study of the process $X_{z,z';\theta}$ and, in fact, $X_{z,z';\theta}$ can be uniquely reconstructed from $\mathcal E_{z,z';\theta}$. However, one should be careful when working with this form: $\mathcal E_{z,z';\theta}$ can only be naturally defined on the functions from a set $\calD[\mathcal E_{z,z';\theta}]$ which coincides with the domain of $\left(-\overline{A}_{z,z';\theta}\right)^{\frac12}$. Here $\overline{A}_{z,z';\theta}$ denotes the closure of $A_{z,z';\theta}$ in $L^2(\Omega,\M)$, which turns out to be nonpositive definite. 

The main technical result of this work describes the restriction of $\mathcal E_{z,z';\theta}$ to the polynomials in natural coordinates $\alpha_i, \beta_j$. Note that $A_{z,z';\theta}$ and $\mathcal E_{z,z';\theta}$ are initially defined in terms of moment coordinates $q_k$, so showing that $\mathcal E_{z,z';\theta}$ can be evaluated at such polynomials already requires a nontrivial argument.

\begin{theo*}[Theorem \ref{result}] Let $\theta=1$. Then $\mathbb R[\alpha_1,\alpha_2,\dots, \beta_1,\beta_2,\dots]\subset \calD[\calE_{z,z';\theta}]$ and we have
\be
\calE_{z,z';\theta}(u,v)=\int_{\Omega} \Gamma_{\alpha\beta}(u,v) d\M,
\ee
where
$$
\Gamma_{\alpha\beta}(u,v)=\sum_{i\geq1} \alpha_i\frac{\partial u}{\partial \alpha_i}\frac{\partial v}{\partial \alpha_i}+\sum_{j\geq 1} \beta_j\frac{\partial u}{\partial \beta_j}\frac{\partial v}{\partial \beta_j}
-\left(\sum_{i\geq1} \alpha_i \frac{\partial u}{\partial \alpha_i}+\sum_{j\geq 1} \beta_j \frac{\partial u}{\partial \beta_j}\right)\left(\sum_i \alpha_i \frac{\partial v}{\partial \alpha_i}+\sum_{j\geq1} \beta_j \frac{\partial v}{\partial \beta_j}\right).
$$
\end{theo*}

Since the form $\Gamma_{\alpha\beta}$ above does not depend on $z,z',\theta$, all dependence of $\calE_{z,z';\theta}$ on these parameters comes from the measure $\M$. Moreover, our result is consistent with the computations of the Dirichlet forms for Ethier-Kurtz's and Petrov's models: if we restrict functions to the subsimplex $\beta_1=\beta_2=\dots=0$ and replace the measure $\M$ by the Poisson-Dirchlet measure $PD(a, \tau)$ we get the form
$$
\int \sum_{i,j}(\1_{i=j}\alpha_i-\alpha_i\alpha_j)\frac{\partial u}{\partial \alpha_i}\frac{\partial v}{\partial \alpha_j}  d PD(a,\tau).
$$
This is exactly the Dirichlet form of Ethier-Kurtz's and Petrov's models computed in \cite{Sch91}, \cite{FS09}.

To prove our result we approximate the natural coordinates $\alpha_i, \beta_j$ by the moment coordinates $q_k$ in such a way that we can track the values of $\calE_{z,z';\theta}$ along this approximation. While the approach itself is not new and it was used in \cite{Sch91},\cite{FS09}, the approximations used earlier do not work for the process $X_{z,z';\theta}$ on Thoma simplex. Instead we have to use Laplace transforms of certain measures $\nu_{\alpha;\beta}$ on the interval $[-\theta;1]$ to connect moment coordinates $q_k$ and natural coordinates $\alpha_i, \beta_j$.

Let us also comment about the restriction $\theta=1$ in our result, which comes from technical limitations and which we believe is not needed. However our proof relies on several facts regarding the measure $\M$ which are currently known only for the case $\theta=1$, where measures $\mathcal{M}_{z,z'}$ can be extensively studied using their connection to Schur functions and symmetric groups, see \cite{BO98} for details. 

The computation of Dirichlet forms allows us to establish the following analogue of \eqref{petface}:
\begin{theo*}[Theorem \ref{boundary-result}] Let $\theta=1$. Then for any point $\omega\in\Omega$ we have
\be
\mathbb P_{\omega}(X_{z,z';\theta}(t)\in\Omega_0, \forall t>0)=1,
\ee
where
$$
\Omega_0=\{(\alpha;\beta)\in\Omega\mid \sum_{i\geq 1}\alpha_i+\sum_{j\geq 1} \beta_j=1\}.
$$
\end{theo*}
Again, we fully prove the statement in the case $\theta=1$, while for general $\theta$ several additional assumptions about $\M$ are needed.

Finally, we try to use our computation of the Dirichlet form to give an expression for the generator $A_{z,z';\theta}$ in terms of the natural coordinates, similarly to \eqref{Pet-gen}. The exact statement can be found in Section \ref{gen-sect} and \eqref{conjA}, but the main feature of our expression is that it involves singularities of the form $(\alpha_i-\alpha_j)^{-1}$ and $(\beta_i-\beta_j)^{-1}$. Such behavior is new, it is not present in Ethier-Kurtz's or Petrov's diffusions, and intuitively it suggests that a repulsive interaction between the coordinates $\alpha_i,\beta_j$ appears when $\theta>0$. This is consistent with the expected repulsive behavior for the measures $\M$, which comes from the similarity to the log-gas models with positive $\beta=2\theta$ (\cite{BO02}). However, the study of the generator $\overline A_{z,z';\theta}$ in terms of the natural coordinates seems to be a much more challenging problem compared to the Dirichlet form, so we leave this discussion as an open question.

\subsection{Notation and layout of the paper} Throughout the text all functions are real-valued, in particular we only work with real Hilbert spaces. We use $C(S)$ to denote the Banach space of real-valued continuous functions on a compact set $S$ with $||\cdot||_{\sup}$-norm. We use $L^2(S,\mu)$ to denote the Hilbert space of square-integrable real-valued function on a space $S$ with a Borel probability measure $\mu$. We use $\langle\cdot,\cdot\rangle_\mu$ to denote the standard scalar product of $L^2(S,\mu)$.

The structure of this paper is as follows. We start by establishing notation regarding Dirichlet forms and Markov processes in Section \ref{background-sect}. Then in Section \ref{process-sect} we provide definitions and known properties of the diffusion processes on Thoma simplex and discuss the assumptions we need for the results of this work. In Section \ref{basicQ-sect} our method of approximating $\alpha_1$ and $\beta_1$ using moment coordinates is described. Section \ref{Ecomp-sect} is dedicated to the proof of our first result, Theorem \ref{result}, which tells how to evaluate the Dirichlet form of the process at natural coordinates. Then, using theory of Dirichlet forms and earlier results, we repeat Ethier's argument to show that the process $X_{z,z';\theta}$ at $t=0$ immediately enters the face $\Omega_0$ and then never leaves it. Finally, in Section \ref{gen-sect} we discuss the description of the generator in terms of natural coordinates.

\subsection{Acknowledgments} The author is deeply grateful to Grigori Olshanski for suggesting this project and explaining the context around it. The author would also like to thank Alexei Borodin, Vadim Gorin and Leonid Petrov for numerous discussions and suggestions regarding this work.


\section{Background on symmetric Markov processes and Dirichlet forms}\label{background-sect}

Here we remind the necessary background about Dirichlet forms and establish our notation for random processes. Our main reference for this material is \cite{FOT10}.

Let $S$ be a compact metrizable space, which we also view as a measurable space equipped with Borel $\sigma$-algebra. We use $D(S)$ to denote the set of right-continuous functions $\omega: \mathbb R_{\geq 0}\to S$ with left limits, such functions are called \emph{c\`adl\`ag functions}. We view $D(S)$ as a measurable space with respect to the minimal $\sigma$-algebra such that $\omega\mapsto\omega(t)$ is measurable for every $t\geq 0$. 

Recall that a Feller process on $S$ is a Markov process with c\`adl\`ag trajectories which can be started from any point $x\in S$ and which satisfies the Feller condition: transition functions depend continuously on the starting point $x$. More precisely, we define a Feller process $X$ on $S$ as the following data:
\begin{itemize}
\item a family of probability measures $\{\P_x\}_{x\in S}$ on $D(S)$,
\item a right-continuous family of filtrations $\{\mathcal F_t\}_{t\geq 0}$ on $D(S)$ such that the variables $X(t):\omega\mapsto \omega(t)$ are adapted with respect to it,
\end{itemize}
satisfying the following properties:
\begin{itemize}
\item for all $x\in S$ we have $\P_x(X(0)=x)=1$,
\item for all $f\in C(S)$ and $t\geq 0$ the mapping $x\mapsto \E_x f(X(t))$ is in $C(S)$,
\item for all $x\in S$ and for all bounded measurable functions $Y$ on $D(S)$ we have the Markov condition 
$$\E_x(Y\circ \theta_s\mid \mathcal F_s)=\E_{X(s)}Y \qquad\P_x \text{\ a.s.},$$
where $\theta_s:\omega(t)\mapsto\omega(t+s)$ is the shift operator. 
\end{itemize}

To each Feller process we can associate the following objects. \emph{Transition functions} $p_t(x, dy)$ are measures on $S$ depending on $x\in S, t\in\R_{\geq0}$ which are defined as distributions of $X(t)$ with respect to $\P_x$. Using $p_t(x,dy)$ as kernels of integral operators we obtain \emph{the transition semigroup} $T_t$. Namely, the operator $T_t$ acts on a measurable bounded function $f$ on $S$ by
$$
[T_tf](x)=\int_S f(y)p_t(x,dy).
$$ 
Since $X$ is a Feller process $T_t C(S)\subset C(S)$. The Feller \emph{generator} of $X$ is the (generally, unbounded) operator $A$ on $C(S)$  defined by
$$
Af=\lim_{t\to 0^+} \frac{T(t)f-f}{t},
$$
where the limit is taken in $C(S)$ and the domain $\calD[A]$ consists of the functions $f$ such that the limit above exists. The generator of a Feller process contains all information about the process and, under certain conditions on $A$, the construction above can be reversed to obtain a Feller process with the given generator $A$. For our purposes we do not need the details of this construction, see \cite[Chapter 3]{Lig10} for details. 

Let $\mu$ be a Borel probability measure on $S$ with topological support $S$, that is, any non-empty open subset $U\subset S$ has positive measure. A process $X$ is called \emph{symmetric} with respect to $\mu$ if for any bounded measurable functions $f, g$ on $S$ we have $\langle T_t f, g\rangle_\mu=\langle f, T_tg\rangle_\mu$.  
In particular, $\mu$ is an invariant measure for the process $X$ since $T_t1\equiv1$. 
\begin{lem} Let $X$ be a symmetric Feller process with respect to $\mu$. Then its generator $A$ has a nonpositive definite self-adjoint extension in $L^2(S, \mu)$.
\end{lem}
\begin{proof} From \cite[(1.4.13) and Lemma 1.4.3]{FOT10} the semigroup $T_t$ can be extended to a strongly continuous semigroup of contracting self-adjoint operators on $L^2(S,\mu)$. Its generator is a nonpositive definite self-adjoint operator by \cite[Lemma 1.3.1]{FOT10}.
\end{proof}
From now on we assume that the self-adjoint extension above is the closure $\overline{A}$ of $A$ in $L^2(\Omega,\M)$, which will be the case for the considered process on Thoma simplex.

Given a nonpositive definite self-adjoint operator $\overline{A}$ on $L^2(S,\mu)$ we can define a symmetric nonnegative definite form $\calE$ with domain $\calD[\calE]=\calD[\sqrt{-\overline{A}}]$ by setting
$$
\calE(u,v)=\langle\sqrt{-\overline{A}}u, \sqrt{-\overline{A}}v\rangle_{\mu}.
$$
Note that if the operator $\overline{A}$ is not bounded then $\calD[\calE]$ is larger than the domain of $\overline{A}$. The resulting form $\calE$ is \emph{closed}, that is, the space $\calD[\calE]$ with the scalar product 
$$
\calE_1(u,v)=\calE(u,v)+\langle u,v\rangle_\mu
$$
is a Hilbert space. In the text we often use the following elementary criterion for $\calE$ to be defined at $f$. 
\begin{prop}
\label{Eprop}
Suppose that $\calE$ is a closed symmetric nonnegative definite form with domain $\calD[\calE]\subset L^2(S,\mu)$, $\mathcal A$ is a set of parameters, $\{f_{\alpha, n}\}_{\alpha\in\mathcal A, n\in\mathbb Z_{\geq 1}}$ is a family of functions from $\calD[\calE]$ which satisfies the following conditions:
\begin{itemize}
\item for each $\alpha\in\mathcal A$ there exists a limit $\lim_{n\to\infty} f_{\alpha,n}=f_\alpha$ in $L^2(S,\mu)$ for $f_\alpha\in L^2(S,\mu)$,
\item for each $\alpha,\beta\in\mathcal A$  there exists a joint limit $\lim_{n,m\to\infty} \calE(f_{\alpha,n}, f_{\beta,m})$.
\end{itemize}
Then for any $\alpha\in\mathcal A$ we have $f_\alpha\in \calD[\calE]$ and for any $\alpha,\beta\in\mathcal A$
\be
\calE(f_{\alpha},f_{\beta})=\lim_{n,m\to\infty}\calE(f_{\alpha,n}, f_{\beta,m}).
\ee
\end{prop}
\begin{proof} 
Since $\lim_{n\to\infty} f_{\alpha,n}=f_\alpha$ in $ L^2(S,\mu)$ we have the joint limits  $\lim_{n,m\to\infty} \langle f_{\alpha,n}, f_{\beta,m}\rangle_{\mu}=\langle f_\alpha,f_\beta\rangle_{\mu}$. Then for each $\alpha\in\mathcal A$ the joint limit $\lim_{n,m\to\infty} \calE_1(f_{\alpha,n}, f_{\alpha,m})$ exists, hence 
$$
\lim_{n,m\to\infty} \calE_1(f_{\alpha,n}-f_{\alpha,m},f_{\alpha,n}-f_{\alpha,m})=0
$$
and $f_{\alpha,n}$ form an $\calE_1$-Cauchy sequence. Since $\D[\calE]$ with $\calE_1$-norm is a Hilbert space, this sequence must converge in $\calE_1$-norm to something, and since  $\calE_1$-norm is stronger than $||\cdot||_{L^2}$-norm the limit must be $f_\alpha$. Hence $f_\alpha\in\D[\calE]$ and, due to $\calE_1$-convergence,
$$
\calE(f_{\alpha},f_{\beta})=\calE_1(f_{\alpha},f_{\beta})-\langle f_{\alpha},f_{\beta}\rangle_\mu=\lim_{n,m\to\infty}\calE_1(f_{\alpha,n}, f_{\beta,m})-\langle f_{\alpha,n},f_{\beta,m}\rangle_\mu=\lim_{n,m\to\infty}\calE(f_{\alpha,n}, f_{\beta,m}).
$$
\end{proof}

When $\overline{A}$ is the closure of a generator of a symmetric Feller process $X$ as above, the resulting form $\calE$ is called the \emph{Dirichlet form} of the process $X$. In this work we only need the following property of Dirichlet forms.

\begin{lem}[{{\cite[Lemma 5.1.2]{FOT10}}}]
\label{uniformLem}
Let $\calE$ be the Dirichlet form of a process $X$ and $\{u_n\}$ be a sequence of continuous functions in $\calD[\calE]$. If $\{u_n\}$ is an $\calE_1$-Cauchy sequence, then there exists a subsequence $\{u_{n_k}\}$ satisfying the following condition for $\mu$ a.e. $x\in S$:
$$
\P_x(u_{n_k}(X_t) \text{\ converges\ uniformly in $t$ on each compact interval of\ } [0,\infty))=1.
$$
\end{lem}


\section{Processes on the Thoma simplex}\label{process-sect}

Here we recall the main results about the family of diffusion processes on the Thoma simplex.

The \emph{Thoma simplex} $\Omega$ is defined as the set of pairs $(\alpha;\beta)$ where $\alpha=(\alpha_i)_{i \geq 1}$ and $\beta=(\beta_i)_{i\geq 1}$ are sequences of nonnegative real numbers satisfying
$$
\alpha_1\geq\alpha_2\geq\alpha_3\geq\dots\geq 0,\qquad \beta_1\geq\beta_2\geq\beta_3\geq\dots\geq 0,
$$
$$
\sum_{i\geq 1}\alpha_i+\sum_{j\geq 1}\beta_j\leq 1.
$$
For a point $(\alpha;\beta)\in\Omega$ we sometimes also use the coordinate $\gamma:=1-\sum_{i\geq 1}\alpha_i-\sum_{j\geq 1}\beta_j$. With the topology of point-wise convergence $\Omega$ is a compact metrizable space, a metric can be explicitly given by
$$
\mathrm{dist}((\alpha; \beta), (\alpha';\beta'))=\sum_{i\geq 1}2^{-i}|\alpha_i-\alpha_i'|+\sum_{j\geq 1}2^{-j}|\beta_j-\beta_j'|.
$$ 

From now on we fix a Jack parameter $\theta>0$\footnote{The parameter $\theta$ comes from the Jack symmetric functions which are closely related to the process we study, see \cite{Olsh09} for details.}. For each point $(\alpha;\beta)\in\Omega$ we can define a probability measure $\nu_{\alpha;\beta}$ on $[-\theta;1]$ by setting 
\begin{equation}\label{thomaMeasuresDef}
\nu_{\alpha;\beta}=\sum_{i\geq 1}\alpha_i\delta_{\alpha_i}+\sum_{j\geq 1}\beta_j\delta_{-\theta\beta_j}+\gamma\delta_0,
\end{equation}
where $\delta_x$ denotes the atomic probability measure at $x$. We call $\nu_{\alpha;\beta}$ the \emph{Thoma measure} corresponding to $(\alpha;\beta)$.

Recall that for integers $k\geq 1$ we define the \emph{moment coordinates} $q_k$ on $\Omega$ by
\be
q_k(\alpha;\beta)=\sum_{i\geq 1}\alpha_i^{k+1}+(-\theta)^k\sum_{j\geq 1}\beta_j^{k+1}=\int_{-\theta}^1x^k\nu_{\alpha;\beta}(dx).
\ee
Note that for each $k\geq 1$ the function $q_k$ is continuous on $\Omega$. Indeed, since $\alpha_1\geq\alpha_2\geq\dots\geq \alpha_n$ and $\alpha_1+\dots+\alpha_n\leq 1$ we have $\alpha_n\leq\frac1n$ and similarly $\beta_n\leq\frac1n$. Then the infinite sums $\sum_{i\geq 1}\alpha_i^{k+1}$ and $\sum_{i\geq 1}\beta_i^{k+1}$ converge uniformly on $\Omega$ since $\sum_{n\geq 1}\frac{1}{n^{k+1}}$ converges, which implies continuity of $q_k$.\footnote{Note that the function $\sum_{i\geq 1}\alpha_i+\sum_{j\geq 1}\beta_j$ is not continuous on $\Omega$ and our argument relies on $k\geq 1$.} Moreover, the functions $q_k$ are algebraically independent on $\Omega$: by setting $\beta_j=0$ the functions $q_k$ become power sums $\sum_{i\geq 1}\alpha_i^k$ which are independent. Let $\Lambda^\circ=\mathbb R[q_1, q_2, \dots]$ denote the algebra generated by $q_k$.

In addition to the parameter $\theta$ we fix a pair of complex parameters $(z,z')$ which satisfies one of the following two conditions: 
\begin{itemize}
\item $z\in \mathbb C\backslash\mathbb R$ and $z'=\overline z$,

\item$\theta$ is rational and both $z$ and $z'$ are real numbers lying in one of the open intervals between two consecutive numbers from the lattice $\mathbb Z+\theta\mathbb Z$.
\end{itemize}
Then we define an operator $A_{z,z';\theta}$ on $\Lambda^\circ$ by
\begin{multline}
\label{Agenequation}
A_{z,z';\theta}=\sum_{k,l\geq 1}(k+1)(l+1)(q_{k+l}-q_kq_l)\frac{\partial^2}{\partial q_k\partial q_l}\\
+\sum_{k\geq 1}(k+1)\left(((1-\theta)k+z+z')q_{k-1}-(k+\theta^{-1}zz')q_k\right)\frac{\partial}{\partial q_k}\\
+\theta\sum_{k,l\geq 0}(k+l+3)q_kq_l\frac{\partial}{\partial q_{k+l+2}}.
\end{multline}
\begin{theo}[{{\cite[Theorems 9.6, 9.7, 9.10]{Olsh09}}}]
(i) The operator $A_{z,z',\theta}$ is closable in $C(\Omega)$ and its closure is a generator of a Feller process $X_{z,z';\theta}$ on $\Omega$.

(ii) The process $X_{z,z';\theta}$ has continuous paths a.s..

(iii) The process $X_{z,z';\theta}$ has a unique symmetrizing measure $\M$.
\end{theo}

The measure $\M$ from the last part is called the  \emph{(boundary) z-measure} and can be described in terms of Jack symmetric functions with Jack parameter $\theta$, see \cite[Section 9.4]{Olsh09} and references there for details. For our purposes we only need the following results regarding this measure.

\begin{theo}[{\cite[Theorem 17]{K18}}] \label{density-theorem} For any $t>0$ the transition function $p_t(x, dy)$ of $X_{z,z';\theta}$ is absolutely continuous with respect to $\M$.
\end{theo}
\begin{theo}[{\cite{Olsh18}}] The topological support of $\M$ is $\Omega$.
\end{theo}

Moreover, for our main results we need to assume that the following two conjectural properties hold:

\begin{con}\label{boundary-assumption} The set 
$$
\{(\alpha;\beta)\in\Omega: \gamma>0\}=\{(\alpha;\beta)\in\Omega: \sum_{i\geq 1} \alpha_i+\sum_{j\geq 1} \beta_j<1\}
$$
is a null set with respect to $\M$.
\end{con}
\begin{con}\label{simple-assumption} For each $i\geq 1$ the sets
$$
\{(\alpha;\beta)\in\Omega: \alpha_i=\alpha_{i+1}>0\}, \qquad \{(\alpha;\beta)\in\Omega: \beta_i=\beta_{i+1}>0\}
$$
are null sets with respect to $\M$.
\end{con}

We expect these statements to be true, however currently they are only known in the case $\theta=1$, namely:

\begin{theo}[{\cite[Theorem 6.1]{Olsh98}, \cite[Theorem 2.5.1]{Bor98}}] Both Conjecture \ref{boundary-assumption} and Conjecture \ref{simple-assumption} hold when $\theta=1$.
\end{theo}

Let us briefly comment why we need these assumptions. In view of Theorem \ref{density-theorem} the first conjecture implies that at any fixed moment $t>0$ we have $X_{z,z';\theta}(t)\in\Omega_0$ a.s., where
$$
\Omega_0=\{(\alpha;\beta)\in\Omega: \gamma=0\}=\{(\alpha;\beta)\in\Omega: \sum_{i\geq 1} \alpha_i+\sum_{j\geq 1} \beta_j=1\}.
$$ 
This is the first step of our proof of Theorem \ref{boundary-result}, which claims that for any $\omega\in\Omega$
\be
\mathbb P_{\omega}(X_{z,z';\theta}(t)\in\Omega_0, \forall t>0)=1.
\ee 
Conjecture \ref{simple-assumption} significantly simplifies our computations of Dirichlet forms in Section \ref{Ecomp-sect}, however we suspect that it is not needed. 

\begin{rem} There is one case when Conjectures \ref{boundary-assumption}, \ref{simple-assumption} are easy to show. Namely, in the regime 
$$
\theta\to 0,\quad z\to 0,\quad z'\to 0, \quad zz'\theta^{-1}\to \tau>0
$$
the measures $\M$ converge weakly to the Poisson-Dirichlet measures $PD(0;\tau)$ on the simplex $\overline{\nabla}_\infty=\{(\alpha,0)\subset\Omega\}$ obtained by setting all $\beta$-coordinates to $0$. Then Conjectures \ref{boundary-assumption}, \ref{simple-assumption} can be seen immediately from the ``stick-breaking" construction of the Poisson-Dirichlet measures (see \cite{PY97}): $PD(0;\tau)$ is the law of the non-increasing infinite sequence consisting of numbers 
$$
\prod_{i=1}^{k-1} (1-Y_i) Y_k, \qquad k\in\mathbb Z_{\geq 1},
$$
where $Y_i$ are independent identically distributed variables with $\mathrm{Beta}(1,\tau)$ distributions. In the general case $\theta>0$ there are no known analogues of ``stick-breaking" construction, making Conjectures \ref{boundary-assumption}, \ref{simple-assumption} significantly harder to approach.
\end{rem}

We finish this section by describing the Dirichlet form of $X_{z,z';\theta}$ in terms of the moment coordinates. By \cite[Theorem 9.9]{Olsh09} the pre-generator $A_{z,z';\theta}$ is diagonalizable with a discrete spectrum in $\mathbb R_{\leq 0}$ and an orthonormal eigenbasis in $L^2(\Omega,\M)$,  so the closure $\overline A$ is self-adjoint and the discussion from Section \ref{background-sect} applies in our setting. Let $\calE_{z,z';\theta}$ denote the corresponding Dirichlet form. Note that the domain $\D[\calE_{z,z';\theta}]$ of $\calE_{z,z';\theta}$ contains $\D[A_{z,z';\theta}]=\Lambda^{\circ}$. Define a bi-differential operator $\Gamma$ on $\Lambda^\circ$ by setting
\begin{equation}\label{gammaOperator}
\Gamma(u,v)=\sum_{k,l\geq 1}(k+1)(l+1)(q_{k+l}-q_kq_l)\frac{\partial u}{\partial q_k}\frac{\partial v}{\partial q_l}.
\end{equation}
\begin{prop}\label{gammaForm}
For $u,v\in\Lambda^{\circ}$ we have
$$
\calE_{z,z';\theta}(u,v)=\int_{\Omega}\Gamma(u,v)d\M.
$$
\end{prop}
\begin{proof}
For $\theta=1$ this is \cite[Theorem 7.4]{BO07} and the general case $\theta>0$ is discussed in \cite[Remark 9.11]{Olsh09}. For completeness we repeat the proof here.

Using the product rule for derivatives note that for any $u,v\in\Lambda^{\circ}$ we have
$$
A_{z,z';\theta}(uv)-(A_{z,z';\theta}u)v-u(A_{z,z';\theta}v)=2\Gamma(u,v).
$$
Taking integral with respect to $\M$ of both sides we get
$$
\langle A_{z,z';\theta}(uv), 1\rangle_{\M}-\langle A_{z,z';\theta}u, v\rangle_{\M}-\langle u, A_{z,z';\theta}v\rangle_{\M}=2\int_{\Omega}\Gamma(u,v)d\M.
$$
Since $A_{z,z';\theta}$ is a symmetric operator in $L^2(\Omega,\M)$, we have 
$$\langle A_{z,z';\theta}(uv), 1\rangle_{\M}=\langle uv, A_{z,z';\theta}1\rangle_{\M}=0$$
$$
-\langle A_{z,z';\theta}u, v\rangle_{\M}=-\langle u, A_{z,z';\theta}v\rangle_{\M}=\langle\sqrt{-\overline{A}_{z,z';\theta}}u, \sqrt{-\overline{A}_{z,z';\theta}}v\rangle_{\M}=\calE(u,v).
$$
So we get $2\calE(u,v)=2\int_{\Omega}\Gamma(u,v)d\M$.
\end{proof}


\begin{section}{Functions $\Q[\phi]$ and $\chi_C$} \label{basicQ-sect}

In this section we introduce functions $\chi_C$ approximating $\alpha_1$ and $\beta_1$. These approximations will play a key role in Theorem \ref{result}, which describes $\calE_{z,z';\theta}$ in terms of natural coordinates.

In what follows $B(X)$ denotes the Banach space of bounded functions on a space $X$ with the supremum norm. To any function $\phi(t)\in B([-\theta;1])$ we can assign a function $\calQ[\phi]$ on $\Omega$ by setting
$$
\calQ[\phi](\alpha;\beta)=\int_{-\theta}^1 \phi(t)\nu_{\alpha;\beta}(dt)=\sum_{i\geq 1}\alpha_i\phi(\alpha_i)+\sum_{j\geq 1}\beta_j\phi(-\theta\beta_j)+\gamma\phi(0),
$$
where $\nu_{\alpha;\beta}$ are Thoma measures from \eqref{thomaMeasuresDef} and $\gamma=1-\sum_{i\geq1}\alpha_i-\sum_{j\geq 1}\beta_j$. Since $\phi$ is bounded $\calQ[\phi](\alpha;\beta)$ is well-defined and since $\nu_{\alpha;\beta}([-\theta;1])=1$ we have 
\begin{equation}
\label{Qnorm}
||\mathcal Q[\phi]||_{\sup}\leq ||\phi||_{\sup},
\end{equation}
where the first norm is the supremum norm in $B(\Omega)$ and the second norm is the supremum norm on $B([-\theta; 1])$. Note that for the moment coordinates we have $q_k=\Q[t^{k}]$.

\begin{prop}
\label{Qcont}
The operator $\Q:B([-\theta; 1])\to B(\Omega)$ has norm $1$ and restricts to $\Q:C([-\theta; 1])\to C(\Omega)$.
\end{prop}
\begin{proof}
The statement about norm is equivalent to $\eqref{Qnorm}$ combined with $\Q[1]=1$. So it is enough to prove that $\Q[\phi]$ is continuous for $\phi\in C([-\theta;1])$.
 
Take $\phi\in C([-\theta;1])$. By Weierstrass theorem there is a sequence of polynomials $p_n\in\mathbb R[t]$ converging uniformly on $[-\theta;1]$ to $\phi$. For any $k\geq 1$ we have $\mathcal Q[t^k]=q_{k}\in C(\Omega)$ and $\Q[1]=1\in C(\Omega)$ for $k=0$. Hence $\Q[p_n]$ are continuous functions uniformly converging to $\Q[\phi]$, thus $\Q[\phi]\in C(\Omega)$.
\end{proof}

Our approximations will be given in terms of functions $\Q[p(t)e^{Ct}]$ where $p(t)$ is a polynomial. To work with these functions we first prove several technical results.

\begin{lem}
\label{exasympt}
Let $p(t)\in\mathbb R[t]$. Then for any $\omega=(\alpha;\beta)\in\Omega$ we have
\be
\lim_{C\to\infty}e^{-C \alpha_{1}}\Q[p(t)e^{Ct}](\omega)=p(\alpha_1)\nu_{\alpha;\beta}(\{\alpha_1\}),
\ee 
\be
\lim_{C\to-\infty}e^{C \theta \beta_{1}}\Q[p(t)e^{Ct}](\omega)=p(-\theta\beta_1)\nu_{\alpha;\beta}(\{-\theta\beta_1\}).
\ee
In particular, if $\alpha_{1}>\alpha_{2}$ we have 
\be
\lim_{C\to\infty}e^{-C \alpha_{1}}\Q[p(t)e^{Ct}](\omega)=\alpha_{1}p(\alpha_{1}),
\ee
and if $\beta_{1}>\beta_{2}$ then
\be
\lim_{C\to-\infty}e^{C \theta \beta_{1}}\Q[p(t)e^{Ct}](\omega)=\beta_{1}p(-\theta\beta_{1}).
\ee
\end{lem}
\begin{proof} We only consider the case $C\to\infty$, the case $C\to-\infty$ is similar.

Note that for fixed $\omega=(\alpha;\beta)$ the measure $\nu_{\alpha;\beta}$ is supported by $[-\theta\beta_1, \alpha_1]$. So if $\{\phi_C(t)\}$ is a family of continuous functions on $[-\theta\beta_1, \alpha_1]$
such that 
\begin{itemize}
\item $\phi_C(t)$ are bounded by a uniform constant independent of $C$ or $t$,
\item $\phi_C(t)$ converge pointwise to $\phi(t)$ as $C\to\infty$,
\end{itemize} 
then by the dominated convergence theorem $\Q[\phi_C](\omega)\to\Q[\phi](\omega)$ as $C\to\infty$.

For any $t\leq \alpha_1$ and $C>0$ we have
\be
|p(t)e^{C(t-\alpha_{1})}|\leq |p(t)|, \quad \lim_{C\to\infty}p(t)e^{C(t-\alpha_{1})}=p(t)\1_{t=\alpha_1},
\ee 
where $\1_{t=\alpha_1}$ is the indicator function of $\{\alpha_1\}$. Hence
\be
\lim_{C\to\infty}e^{-C \alpha_{1}}\Q[p(t)e^{Ct}](\alpha;\beta)=\lim_{C\to\infty}\Q[p(t)e^{C(t-\alpha_1)}](\alpha;\beta)=\Q[p(t)\1_{t=\alpha_{1}}](\alpha;\beta)=p(\alpha_1)\nu_{\alpha;\beta}(\{\alpha_1\}).
\ee
The second claim follows immediately since for $\alpha_1>\alpha_2$ we have $\nu_{\alpha;\beta}(\{\alpha_1\})=\alpha_1$.
\end{proof}

\begin{lem}
\label{expineq}
Fix $\omega=(\alpha;\beta)\in\Omega$. For $k\in\mathbb Z_{\geq 0}$ and $C\in\mathbb R_{>0}$ the following inequalities hold
\be
|\Q[t^ke^{Ct}](\omega)|\leq \max(\alpha_{1}^ke^{\alpha_{1}C}; \theta^k),
\ee
\be
|\Q[t^ke^{-Ct}](\omega)|\leq \max(\theta^k\beta_{1}^ke^{\theta\beta_{1}C}; 1).
\ee
Moreover, we have 
\be
\alpha_{1}e^{\alpha_{1}C} \leq\Q[e^{Ct}](\omega), \quad \beta_{1}e^{\theta\beta_{1}C} \leq\Q[e^{-Ct}](\omega).
\ee
\end{lem}
\begin{proof}

Recall that $\nu_{\alpha;\beta}$ is supported by $[-\theta\beta_{1}, \alpha_{1}]$, hence for any $\phi(t)\in C[-\theta\beta_{1},\alpha_{1}]$ we have
\be
\left |\Q[\phi(t)](\omega)\right|\leq \norm{\phi}_{C[-\theta\beta_{1},\alpha_{1}]}.
\ee 
Note that $t^ke^{Ct}$ is increasing on $[0, \alpha_{1}]$, hence for $t\in [0,\alpha_{1}]$ we have
\be
t^ke^{Ct}\leq \alpha_{1}^ke^{C\alpha_{1}},
\ee
\be
t^ke^{-Ct}\leq t^k\leq 1.
\ee
For $t\in[-\theta\beta_{1},0]$ both $|t^k|$ and $e^{-Ct}$ are decreasing, hence
\be
|t^ke^{Ct}|\leq |t^k|\leq \theta^k,
\ee
\be
|t^ke^{-Ct}|\leq \theta^k\beta_{1}^ke^{\theta\beta_{1}C}.
\ee
So 
\be
\left |\Q[t^ke^{Ct}](\omega)\right |\leq \norm{t^ke^{Ct}}_{C[-\theta\beta_1,\alpha_1]}\leq \max(\alpha_1^ke^{\alpha_1C},\theta^k),
\ee
\be
\left |\Q[t^ke^{-Ct}](\omega)\right |\leq \norm{t^ke^{-Ct}}_{C[-\theta\beta_1, \alpha_1]}\leq \max(\theta^k\beta^k_1e^{\theta\beta_1C},1),
\ee
For the last statement note that $e^{\pm Ct}>0$ for any $t$, hence
\be
\Q[e^{Ct}]=\sum_{i\geq 1}\alpha_ie^{C\alpha_i}+\sum_{j\geq 1}\beta_je^{-\theta C\beta_{j}}+\gamma\geq \alpha_1e^{C\alpha_1},
\ee
\be
\Q[e^{-Ct}]=\sum_{i\geq 1}\alpha_ie^{-C\alpha_i}+\sum_{j\geq 1}\beta_je^{\theta C\beta_{j}}+\gamma\geq \beta_1e^{\theta C\beta_1}.
\ee
\end{proof}

Now we are ready to introduce our approximation of $\alpha_1$ and $\beta_1$. Define $\chi_C$ by
\be
\chi_C=C^{-1}\ln(1+\Q[e^{Ct}]).
\ee
The goal of this section is to show that these functions converge to $\alpha_1$ and $\beta_1$ in $L^2(\Omega,\M)$. To do so we use dominated convergence theorem, showing a uniform bound and a pointwise convergence.

\begin{prop}\label{chiBound}
The functions $\chi_C$  are uniformly bounded for all $|C|\geq 1$ by $1+\theta+\ln2$.
\end{prop}
\begin{proof}
By Lemma \ref{expineq} we have
$$
\norm{\Q[e^{Ct}]}_{\sup}\leq \max(1; e^{|C|\alpha_1}; e^{\theta|C|\beta_1})\leq1+e^{|C|(1+\theta)},
$$
where the last inequality holds since $\alpha_1,\beta_1\leq 1$. Then
\be
\norm{C^{-1}\ln(1+\Q[e^{Ct}])}_{\sup}\leq |C|^{-1}\ln(2+e^{|C|(1+\theta)})\leq |C|^{-1}\ln(2e^{|C|(1+\theta)})\leq 1+\theta+\ln2.
\ee
\end{proof}

\begin{prop}\label{chiLimit}
The functions $\chi_{C}$ converge pointwise to $\alpha_{1}$ when $C\to\infty$ and to $-\theta \beta_{1}$ when $C\to-\infty$.
\end{prop}
\begin{proof}

Fx $\omega=(\alpha,\beta)\in\Omega$. We have
\begin{multline*}
\chi_C(\omega)=C^{-1}\ln(1+\Q[e^{tC}](\omega))=C^{-1}(C\alpha_1+\ln(e^{-C\alpha_1}+e^{-C\alpha_1}\Q[e^{tC}](\omega)))=\\\alpha_1+C^{-1}\ln(e^{-C\alpha_1}+e^{-C\alpha_1}\Q[e^{tC}](\omega)).
\end{multline*}
By Lemma \ref{exasympt} we have $\lim_{C\to\infty}e^{-C\alpha_1}\Q[e^{tC}](\omega)=\nu_{\omega}(\{\alpha_1\})$. Now we have two cases:

{\bfseries Case 1:} $\alpha_1=0$. Then $\nu_{\omega}(\{\alpha_1\})=\gamma$ and we have
$$
\lim_{C\to\infty}\chi_C(\omega)=C^{-1}\ln(1+\gamma)=0=\alpha_1.
$$

{\bfseries Case 2:} $\alpha_1>0$. Then $\nu_{\omega}(\{\alpha_1\})=m\alpha_1>0$ where $m$ is the number of coordinates $\alpha_i$ equal to $\alpha_1$. Then
\be
\lim_{C\to\infty}\ln(e^{-C\alpha_1}+e^{-C\alpha_1}\Q[e^{tC}](\omega))=\ln(m\alpha_1)
\ee
and
\be
\lim_{C\to\infty}\chi_C(\omega)=\alpha_1+\lim_{C\to\infty}C^{-1}\ln(e^{-C\alpha_1}+e^{-C\alpha_1}\Q[e^{tC}](\omega)) =\alpha_1.
\ee

Similarly we have
\begin{multline*}
\chi_C(\omega)=C^{-1}\ln(1+\Q[e^{tC}](\omega))=C^{-1}(-C\theta\beta_1+\ln(e^{C\theta\beta_1}+e^{C\theta\beta_1}\Q[e^{tC}](\omega)))=\\-\theta\beta_1+C^{-1}\ln(e^{C\theta\beta_1}+e^{C\theta\beta_1}\Q[e^{tC}](\omega))\to -\theta\beta_1
\end{multline*}
as $C\to-\infty$.
\end{proof}

\end{section}

\begin{section}{The value of $\calE_{z,z';\theta}$ at natural coordinates.}\label{Ecomp-sect}

In this section we evaluate the Dirichlet form $\calE_{z,z';\theta}$ of $X_{z,z';\theta}$ at the natural coordinates $\alpha_i$ and $\beta_j$.

\subsection{Outline of the computation.}  To demonstrate our argument let us first explain how to evaluate $\calE_{z,z';\theta}$ at $\alpha_1$. Our starting point is Proposition \ref{gammaForm}, which tells how to compute $\calE_{z,z';\theta}(u,v)$ for $u,v\in\Lambda^\circ$. Since $\calE_{z,z';\theta}$ is closed, to extend this computation to $\alpha_1$ and to show that $\alpha_1\in\calD[\calE_{z,z';\theta}]$ it is enough to exhibit a $\calE_1$-Cauchy sequence $f_n\in \mathbb R[q_1,q_2,\dots]$ converging to $\alpha_1$. We can do it in two steps: first we show that the functions $\chi_C$ from the previous section are in $\calD[\calE_{z,z';\theta}]$, then we show that $\chi_C$ is an $\calE_1$-Cauchy sequence converging to $\alpha_1$ as $C\to\infty$. Similarly we can deal with $\beta_1$ by taking $C\to-\infty$ instead.

However, the approach above is not enough to deal with the remaining coordinates $\alpha_i,\beta_j$, since the functions $\chi_C$ approximate only $\alpha_1$ and $\beta_1$. To overcome this issue we note that shifting the $\alpha$-coordinates by one step to the left replaces the moment coordinates $q_k$ with
 $$
 q_k^{(1,0)}:=q_k-\alpha_1^{k+1}=\sum_{i\geq 2}\alpha_2^{k+1}+(-\theta)^k\sum_{j\geq 1}\beta_j^{k+1}.
 $$ 
Using our knowledge about the values of $\calE_{z,z';\theta}$ at $\alpha_1$ and $q_k$ we can compute $\calE_{z,z';\theta}$ at these new moment coordinates and then, using $q_k^{(1,0)}$ instead of $q_k$ to construct $\chi_C$, we get suitable approximations of $\alpha_2$. This idea leads to the inductive computation presented below.

Let $\tau_{\alpha}$ and $\tau_\beta$ denote continuous endomorphisms of $\Omega$ defined by  
\be
\tau_{\alpha}:(\alpha;\beta)\mapsto(\alpha_2,\alpha_3,\dots; \beta_1, \beta_2,,\dots), \quad \tau_{\beta}:(\alpha;\beta)\mapsto(\alpha_1,\alpha_2,\dots; \beta_2,\beta_3,\dots).
\ee
In other words, $\tau_\alpha$ shifts $\alpha$-coordinates one step to the left, while $\tau_\beta$ does the same for $\beta$-coordinates. These maps induce contracting operators $\tau_\alpha^*,\tau_\beta^*: B(\Omega)\to B(\Omega)$ which preserve continuity of functions. For integers $N,M\geq0$ and a function $f\in B(\Omega)$ we use $f^{(N,M)}$ to denote $(\tau_\alpha^*)^N(\tau_\beta^*)^Mf$ and $\Q^{(N,M)}[\phi]$ to denote $(\tau_\alpha^*)^N(\tau_\beta^*)^M\Q[\phi]$. In particular
\be
q^{(N,M)}_k=\Q^{(N,M)}[t^k]=\sum_{i\geq N+1}\alpha_i^{k+1}+(-\theta)^k\sum_{j\geq M+1}\beta_j^{k+1}.
\ee
For the remainder of the section let $\lan f\ran_{\M}$ denote the integration with respect to $\M$. To simplify expressions we also sometimes denote $\alpha_i$ by $x_i$ and $\beta_j$ by $x_{-j}$, assuming that $x_0=\frac{\partial}{\partial x_0}=0$. 

Let $\mathcal P^{(N,M)}$ denote the subalgebra of $C(\Omega)$ generated by 
\begin{equation}\label{NMgen}
\{\alpha_1, \dots, \alpha_N, \beta_1,\dots,\beta_M, q_1^{(N,M)}, q_2^{(N,M)}, \dots\}=\{x_{-M},\dots,x_{N},q_1^{(N,M)},q_2^{(N,M)},\dots\}.
\end{equation}
\begin{lem}
The generators of $\calP^{(N,M)}$ listed in \eqref{NMgen} are algebraically independent.
\end{lem}
\begin{proof} Proof by induction on $N,M$. The base case $N=M=0$ follows from the algebraic independence of power sums $x_1^k+x_2^k+\dots$ in infinitely many variables. For the inductive step assume that we have proved the claim for $\calP^{(N,M)}$, let us prove it for $\calP^{(N+1,M)}$, the claim for $\calP^{(N,M+1)}$ is proved identically. 

Assume that there is a nontrivial algebraic dependence $\Psi$ in $\calP^{(N+1,M)}$, that is $\Psi$ is a polynomial expression in $\{\alpha_i\}_{i\in[1;N+1]}$, $\{\beta_j\}_{i\in[1;M]}$ and $\{q^{(N+1,M)}_k\}_{k\geq 1}$ such that $\Psi(\omega)=0$  for any $\omega\in\Omega$. Note that for any $\omega=(\alpha;\beta)\in\Omega$ and any $\lambda\in(0,1)$ we can construct a point $\omega_{\lambda}=(\lambda\alpha_1, \lambda\alpha_2,\dots; \lambda\beta_1, \lambda\beta_2,\dots)\in\Omega$. Since $\Psi$ vanishes at such $\omega_\lambda$ for any $\lambda$, each homogenous component of $\Psi$ with respect to the natural grading $\deg(\alpha_i)=\deg(\beta_j)=1$, $\deg\left(q^{(N+1,M)}_k\right)=k+1$ vanishes at any $\omega\in\Omega$. So we can assume that $\Psi$ is homogenous with respect to this grading.

 At the same time for any $\omega=(\alpha;\beta)\in\Omega$ and a scalar $\mu\in(0,1)$ we can construct another point $\omega'_{\mu}=(\alpha_1, \mu\alpha_2,\dots$; $\mu\beta_1, \mu\beta_2,\dots)\in\Omega$ where $\alpha_1$ is unchanged, and the vanishing of $\Psi$ at such points for any $\mu$ leads to a nontrivial dependence $\Psi$ which is in addition homogenous with respect to the modified grading where $\deg(\alpha_1)=0$ but all other degrees are the same. This is only possible if $\Psi$ has the form $\alpha_1^k\tau_\alpha^*\Phi$ where $\Phi\in\calP^{(N,M)}$ is homogenous of some degree $d$ with respect to the natural grading. But then by inductive assumption there exists a point $\omega=(\alpha;\beta)\in\Omega$ such that $\Phi(\omega)\neq 0$ and taking the point 
 $$
 \tilde\omega=\left(\frac12, \frac{\alpha_1}2, \frac{\alpha_2}2,\dots; \frac{\beta_1}2, \frac{\beta_2}2, \frac{\beta_3}2,\dots\right)
 $$
 gives $\Psi(\tilde\omega)=\Phi(\omega)2^{-k-d}\neq0$, leading to contradiction.
 \end{proof}

Since for each $N,M\in\mathbb Z_{\geq 0}$ the generators of $\calP^{(N,M)}$ are independent, we can define a bilinear form $E^{(N,M)}$ with the domain $\mathcal P^{(N,M)}$ by
\begin{equation}
\label{Evaluedef}
E_{z,z';\theta}^{(N, M)}(u,v)=\lan\Gamma^{(N, M)}(u,v)\ran_{\M},
\end{equation}
where
\begin{multline}\label{Gammaval}
\Gamma^{(N, M)}(u,v)=\sum_{i,j=-M}^N(\1_{i=j} x_i-x_ix_j)\frac{\partial u}{\partial x_i}\frac{\partial v}{\partial x_j}\\
+\sum_{i=-M}^{N}\sum_{k\geq1} -(k+1)x_iq_k^{(N,M)}\left(\frac{\partial u}{\partial x_i}\frac{\partial v}{\partial q_k^{(N,M)}}+\frac{\partial v}{\partial x_i}\frac{\partial u}{\partial q_k^{(N,M)}}\right)\\
+\sum_{k,l\geq1}(k+1)(l+1)(q_{k+l}^{(N,M)}-q_k^{(N,M)}q_l^{(N,M)})\frac{\partial u}{\partial q_k^{(N,M)}}\frac{\partial v}{\partial q_l^{(N,M)}}.
\end{multline}
 
Now we can formulate one of our main results.

\begin{theo}
\label{Evalue}
Let $N,M\in\mathbb Z_{\geq0}$ and assume that $\theta=1$ or $\theta$ takes any other positive value for which Conjecture \ref{simple-assumption} holds. Then we have $\calP^{(N,M)}\subset\calD[\calE_{z,z';\theta}]$ and for any $u,v\in\calP^{(N,M)}$
\begin{equation}
\label{Evalueeq}
\calE_{z,z';\theta}(u,v)=E^{(N,M)}_{z,z';\theta}(u,v).
\end{equation}
\end{theo}

Theorem \ref{Evalue} is proved by induction. Note that for $N=M=0$ we have $\mathcal P^{(0,0)}=\Lambda^\circ$ and $\Gamma^{(0,0)}=\Gamma$, so the indiction base follows from Proposition \ref{gammaForm}. The remainder of this section is devoted to the step of induction, but before it let us state a particular case of the theorem mentioned in the introduction. 

\begin{theo}
\label{result}
Let $\calP_{nat}=\mathbb R[\alpha_1,\alpha_2,\dots, \beta_1,\beta_2,\dots]$ and assume that $\theta=1$ or $\theta$ takes any other positive value for which Conjecture \ref{simple-assumption} holds. Then we have $\calP_{nat}\subset\mathcal{D}[\mathcal E_{z,z';\theta}]$ and for $u,v\in\calP_{nat}$
\be
\calE_{z,z';\theta}(u,v)=\int_{\Omega} \Gamma_{\alpha\beta}(u,v) d\M,
\ee
where $\Gamma_{\alpha\beta}$ is a bi-differential operator defined on $\calP_{nat}$ by
$$
\Gamma_{\alpha\beta}(u,v)=\sum_{i\geq1} \alpha_i\frac{\partial u}{\partial \alpha_i}\frac{\partial v}{\partial \alpha_i}+\sum_{j\geq 1} \beta_j\frac{\partial u}{\partial \beta_j}\frac{\partial v}{\partial \beta_j}
-\left(\sum_{i\geq1} \alpha_i \frac{\partial u}{\partial \alpha_i}+\sum_{j\geq 1} \beta_j \frac{\partial u}{\partial \beta_j}\right)\left(\sum_i \alpha_i \frac{\partial v}{\partial \alpha_i}+\sum_{j\geq1} \beta_j \frac{\partial v}{\partial \beta_j}\right).
$$
\end{theo}
\begin{proof}
Let $u,v\in \calP_{nat}$. Note that both $u,v$ depend only on a finite number of coordinates, hence $u,v\in \calP^{(N,M)}$ for large enough $N,M$. Then we can apply Theorem \ref{Evalue}.
\end{proof}

\subsection{Value of $\calE_{z,z';\theta}$ at $\chi_C$.} For the remainder of this section we are proving the inductive step of Theorem \ref{Evalue}, so from now on assume that the theorem holds for fixed $N,M\geq 0$. Our first goal is to evaluate $\calE_{z,z';\theta}$ at the functions $\chi_C^{(N,M)}$, which is done in this subsection. 

Let $C^1[-\theta;1]$ denote the Banach space of differentiable functions $f$ on $[-\theta;1]$ with a continuous derivative $f'$, where the norm is defined by
\be
\norm{f}_{C^1}=\norm{f}_{\sup}+\norm{f'}_{\sup}.
\ee 
For $\phi,\psi\in C^1[-\theta;1]$ and $u\in\calP^{(N,M)}$ define
\begin{equation}\label{tmpGammaDouble}
\Gamma^{(N,M)}[\phi;\psi]=\Q^{(N,M)}[(t\phi)'(t\psi)']-\Q^{(N,M)}[(t\phi)']\Q^{(N,M)}[(t\psi)'],
\end{equation}
\begin{equation}\label{tmpGammaSingle}
\Gamma^{(N,M)}[\phi](u)=-\sum_{i=-M}^Nx_i\Q^{(N,M)}[(t\phi)'-\phi(0)]\frac{\partial u}{\partial x_i} +\sum_{k\geq 1}(k+1)(\Q^{(N,M)}[(t\phi)'t^k]-\Q^{(N,M)}[(t\phi)']q^{(N,M)}_k)\frac{\partial u}{\partial q^{(N,M)}_k}.
\end{equation}
\begin{prop}
\label{prop111}
Suppose that we have $u,v\in\calP^{(N,M)}$, functions $\phi,\psi\in C^1[-\theta;1]$ and functions $f\in C^1(\im_\phi), g\in C^1(\im_\psi)$, where $\im_\eta$ denotes the image of $\eta: [-\theta;1]\to \mathbb R$. Then 
$$f\left(\Q^{(N,M)}[\phi]\right)u, g\left(\Q^{(N,M)}[\psi]\right)v\in\D[\calE_{z,z';\theta}]$$
and
\begin{multline}
\label{EQexpressiongen}
\calE(f(\Q^{(N,M)}[\phi])u, g(\Q^{(N,M)}[\psi])v)=\langle f(\Q^{(N,M)}[\phi])g(\Q^{(N,M)}[\psi])\Gamma^{(N,M)}(u,v)\\
+f'(\Q^{(N,M)}[\phi])g(\Q^{(N,M)}[\psi])u\Gamma^{(N,M)}[\phi](v)+f(\Q^{(N,M)}[\phi])g'(\Q^{(N,M)}[\psi])v\Gamma^{(N,M)}[\psi](u)\\
+f'(\Q^{(N,M)}[\phi])g'(\Q^{(N,M)}[\psi])uv\Gamma^{(N,M)}[\phi;\psi]\rangle_{\M}
\end{multline}
\end{prop}
The proof is based on polynomial approximation and is split in a number of lemmas. Since $N,M$ are fixed,  for the duration of this proof we replace the superscript $(N,M)$ by $\sim$, writing $\Gamma^\sim, q^\sim_k, \calQ^\sim[\phi],\calP^\sim$ instead of $\Gamma^{(N,M)}$, $q_k^{(N,M)}$, $\calQ^{(N,M)}[\phi]$ and $\calP^{(N,M)}$. 

\begin{lem}
\label{lem000}
For any bounded $\phi\in B([-\theta;1])$ we have $||\calQ^\sim[\phi]||_{\sup}\leq||\phi||_{\sup}$.
\end{lem}
\begin{proof}
Immideatly follows from Proposition \ref{Qcont} since $\calQ^{(N,M)}[\phi]$ is the restriction of the unshifted $\calQ[\phi]$ to $\tau^N_\alpha\tau^M_\beta\Omega$.
\end{proof}

\begin{lem}
\label{lem111}
For polynomials $\phi, \psi\in\R[t]$ we have
\be
\Gamma^\sim(\Q^\sim[\phi], \Q^\sim[\psi])=\Gamma^\sim[\phi,\psi],
\ee
where the right-hand side is defined by \eqref{tmpGammaDouble}.
\end{lem}
\begin{proof}
By linearity it is enough to consider $\phi=t^k,\psi=t^l$. Moreover, using symmetry we may assume $k\leq l$.  

For $l\geq k\geq 1$ we have
\begin{multline*}
\Gamma^\sim(\mathcal Q^\sim[t^{k}],\mathcal Q^\sim[t^{l}])=\Gamma^\sim(q^\sim_k,q^\sim_l)=(k+1)(l+1)(q^{\sim}_{k+l}-q^\sim_kq^\sim_l)\\
=\mathcal Q^\sim[(t^{k+1})'(t^{l+1})']-\mathcal Q^\sim[(t^{k+1})']\mathcal Q^\sim[(t^{l+1})']=\Gamma^\sim[t^k,t^l].
\end{multline*}

For $k=0$ note that $\Q^\sim[t^k]\equiv 1$ and
\be
\Gamma^\sim(\Q^\sim[t^k],\Q^\sim[t^l])=0=\Q^\sim[(t^{l+1})']-\Q^\sim[(t^{l+1})']=\Q^\sim[(t^{k+1})'(t^{l+1})']-\Q^\sim[(t^{k+1})']\Q^\sim[(t^{l+1})'].
\ee
\end{proof}

\begin{lem}
\label{lem222}
For a polynomial $\phi\in\R[t]$ and $u\in\calP^\sim$ we have
\be
\Gamma^\sim(\mathcal Q^\sim[\phi], u)=\Gamma^\sim[\phi](u),
\ee
where the right-hand side is defined by \eqref{tmpGammaSingle}.
\end{lem}
\begin{proof}
Both sides are differential operators in $u$, so it is enough to consider generators of $\calP^\sim$, namely, $u=x_i$ for $-M\leq i\leq N$ and $u=q^\sim_k$ for $k\geq 1$. Similarly by linearity with respect to $\phi$ it is enough to consider $\phi=t^l$. 

If $u=x_i$ and $\phi=t^l$ for $l\geq 1$ we have
\be
\Gamma^\sim(\mathcal Q^\sim[t^{l}],x_i)=\Gamma^\sim(q^\sim_l,x_i)=-(l+1)x_iq^\sim_l=-x_i\mathcal Q^\sim[(t\phi)'-\phi(0)],
\ee
while for $u=x_i$ and $\phi=1$ we have
\be
\Gamma^\sim(\mathcal Q^\sim[\phi],x_i)=0=-x_i\mathcal Q^\sim[(t\phi)'-\phi(0)].
\ee
The remaining case $u=q^\sim_k=\Q^\sim[t^k]$ follows directly from Lemma \ref{lem111}:
$$
\Gamma^\sim(\mathcal Q^\sim[\phi], \Q^\sim[t^k])=\Gamma^\sim[\phi,t^k]=\Gamma^\sim[\phi](q_k^\sim),
$$
\end{proof}

\begin{lem}
\label{Gammaextcont} $(\psi,\phi)\mapsto\Gamma^\sim[\psi,\phi]$ is a continuous map $C^1[-\theta;1]\times C^1[-\theta;1]\to C(\Omega)$ and for any $u\in\calP^\sim$
the assignment $\psi\mapsto\Gamma^\sim[\psi](u)$ is a continuous map $C^1[-\theta;1]\to C(\Omega)$.
\end{lem}
\begin{proof}
Note that on $[-\theta;1]$ we have
\be
||(t\phi)'||_{\sup}\leq||t\phi'||_{\sup}+||\phi||_{\sup}\leq(1+\theta)||\phi'||_{\sup}+||\phi||_{\sup}\leq(1+\theta)||\phi||_{C^1}.
\ee
Also recall that $||\calQ^\sim[\phi]||_{\sup}\leq||\phi||_{\sup}$ by Lemma \ref{lem000}. Then for any $(\phi,\psi)\in C^1[-\theta;1]\times C^1[-\theta;1]$ we have
\be
||\Gamma^\sim[\psi;\phi]||_{\sup}\leq ||(t\phi)'(t\psi)'||_{\sup}+||(t\phi)'||_{\sup}||(t\psi)'||_{\sup}\leq2||(t\phi)'||_{\sup}||(t\psi)'||_{\sup}\leq 2(1+\theta)^2||\psi||_{C^1}||\phi||_{C^1}.
\ee
This implies the first continuity.

Similarly we have
\begin{multline*}
||\Gamma^\sim[\phi](u)||_{\sup}\leq \sum_{i=-M}^N\norm{(t\phi)'-\phi(0)}_{\sup}\norm{x_i\frac{\partial u}{\partial x_i}}_{\sup}
+\sum_{k\geq 1}2(k+1)\norm{(t\phi)'}_{\sup}\norm{t^k}_{\sup}\norm{\frac{\partial u}{\partial q_k}}_{\sup}\\
\leq K(\norm{(t\phi)'}_{\sup}+|\phi(0)|)\leq 2K(1+\theta)\norm{\phi}_{C^1},
\end{multline*}
where $K$ is a constant depending only on $u$ (note that the second sum in the middle expression has finitely many non-zero terms since $u\in\calP^\sim$).
\end{proof}

\begin{proof}[Proof of Proposition \ref{prop111}]
Fix $u,v\in\calP^\sim$ from the statement and for $\phi,\psi\in C^1[-\theta;1], f\in C^1(\im_\phi), g\in C^1(\im_\psi)$ let $\Psi(\phi,\psi,f,g)$ denote the right-hand side of \eqref{EQexpressiongen}. 

First assume that $\phi,\psi,f,g$ are polynomials. Then we have $f(\calQ^\sim[\phi]),g(\calQ^\sim[\psi])\in\calP^\sim$ and by the inductive assumption
$$
\calE_{z,z';\theta}(f(\calQ^\sim[\phi])u,g(\calQ^\sim[\psi])v)=\left\langle\Gamma^\sim(f(\calQ^\sim[\phi])u,g(\calQ^\sim[\psi])v)\right\rangle_{\M}.
$$
Since $\Gamma^\sim$ is a bi-differential operator we have 
\begin{multline*}
\Gamma^\sim(f(\calQ^\sim[\phi])u,g(\calQ^\sim[\psi])v)=f(\calQ^\sim[\phi])g(\calQ^\sim[\psi])\Gamma^\sim(u,v)+\\
f'(\calQ^\sim[\phi])g(\calQ^\sim[\psi])u\Gamma^\sim(\calQ^\sim[\phi],v)+f(\calQ^\sim[\phi])g'(\calQ^\sim[\psi])v\Gamma^\sim(u,\calQ^\sim[\psi])+\\
f'(\calQ^\sim[\phi])g'(\calQ^\sim[\psi])uv\Gamma^\sim(\calQ^\sim[\phi],\calQ^\sim[\psi]).
\end{multline*}
Then Lemmas \ref{lem111} and \ref{lem222} imply \eqref{EQexpressiongen} for this case.

Now take $\phi,\psi\in C^1[-\theta;1]$ and polynomials $f,g\in\R[t]$. By Weierstrass theorem there exist sequences of polynomials $\phi_n, \psi_m$ converging in $C^1[-\theta;1]$ to $\phi,\psi$ respectively. These polynomials are uniformly bounded by some constant $R$, i.e. $\norm{\phi_n}_{\sup}, \norm{\psi_m}_{\sup}\leq R$. Then, since $||\calQ^\sim[\phi_n]||_{\sup}\leq R$ by Lemma \ref{lem000}, we know that $\calQ^\sim[\phi_n]$ and $\calQ^\sim[\phi]$ have values in $[-R,R]$, and similarly for $\psi_m,\psi$. Thus, using the facts that $\calQ^\sim$ is contracting in supremum norm and $f,g,f',g'$ are uniformly continuous on $[-R,R]$, we get the following set of uniform convergences on $\Omega$:
\be
f(\Q^\sim[\phi_n])\rightrightarrows f(\Q^\sim[\phi]), \quad g(\Q^\sim[\psi_m])\rightrightarrows g(\Q^\sim[\psi]),
\ee
\be
f'(\Q^\sim[\phi_n])\rightrightarrows f'(\Q^\sim[\phi]), \quad g'(\Q^\sim[\psi_m])\rightrightarrows g'(\Q^\sim[\psi]).
\ee
On the other hand Lemmas \ref{lem111}-\ref{Gammaextcont} imply that 
\be
\Gamma^\sim(\Q^\sim[\phi_n],\Q^\sim[\psi_m])=\Gamma^\sim[\phi_n,\psi_m]\rightrightarrows\Gamma^\sim[\phi,\psi],
\ee 
\be
\Gamma^\sim(\Q^\sim[\phi_n],v)=\Gamma^\sim[\phi_n](v)\rightrightarrows\Gamma^\sim[\phi](v),\qquad \Gamma^\sim(\Q^\sim[\psi_n],u)=\Gamma^\sim[\psi_n](u)\rightrightarrows\Gamma^\sim[\psi](u).
\ee 
Then the integrands of $\Psi(\phi_n,\psi_m,f,g)$ converge uniformly to the integrand of $\Psi(\phi,\psi,f,g)$ as $n,m\to\infty$, so the integrals themselves also converge.  But we already know \eqref{EQexpressiongen} for polynomials hence we get
\be
\calE_{z,z';\theta}(f(\Q^\sim[\phi_n])u, g(\Q^\sim[\psi_m])v)=\Psi(\phi_n,\psi_m,f,g)\to \Psi(\phi,\psi,f,g).
\ee
By Proposition \ref{Eprop} this convergence implies that $f(\Q^\sim[\phi])u,g(\Q^\sim[\psi])v\in\D[\calE_{z,z';\theta}]$ and \eqref{EQexpressiongen} holds for this case.

Finally, assume that $\phi,\psi\in C^1[-\theta;1]$ and $f\in C^1(\im_\phi), g\in C^1(\im_\psi)$. The functions $\phi,\psi$ are continuous functions on $[-\theta;1]$ hence their images $\im_\phi,\im_\psi$ are compact intervals in $\R$. Note that $\Q^{(N,M)}[\phi]$ can be viewed as the weighted avarage
$$
\Q^{(N,M)}[\phi]=\sum_{i\geq N+1}\alpha_i\phi(\alpha_i)+\sum_{j\geq M+1}\beta_j\phi(\beta_j)+(1-\!\sum_{i\geq N+1}\!\alpha_i-\!\sum_{j\geq M+1}\!\beta_j)\phi(0)
$$
and since $\im_\phi$ is closed and convex this implies that the values of $\calQ^\sim[\phi]$ are in $\im_\phi$, and similarly for $\calQ^\sim[\psi]$. So the functions $f(\calQ^\sim[\phi]), g(\calQ^\sim[\psi])$ are well-defined. 

Now take sequences of polynomials $f_n,g_m$ converging to $f,g$ in $C^1(\im_\phi)$ and $C^1(\im_\psi)$ respectively. Then we have uniform convergences
\be
f_n(\Q^\sim[\phi])\rightrightarrows f(\Q^\sim[\phi]), \quad g_m(\Q^\sim[\psi])\rightrightarrows g(\Q^\sim[\psi]),
\ee
\be
f_n'(\Q^\sim[\phi])\rightrightarrows f'(\Q^\sim[\phi]), \quad g_m'(\Q^\sim[\psi])\rightrightarrows g'(\Q^\sim[\psi]).
\ee
Hence $\Psi(\phi,\psi,f_n,g_m)$ converges to $\Psi(\phi,\psi,f,g)$ as $n,m\to\infty$. By the previous case we already know that 
$$
\calE_{z,z';\theta}(f_n(\Q^\sim[\phi])u, g_m(\Q^\sim[\psi])v)=\Psi(\phi,\psi,f_n,g_m).
$$
So we can again apply Proposition \ref{Eprop} to get that $f(\Q^\sim[\phi])u,g(\Q^\sim[\psi])v\in\D[\calE_{z,z';\theta}]$ and \eqref{EQexpressiongen} holds.
\end{proof}

Using Proposition \ref{prop111} we can now evaluate $\calE_{z,z';\theta}$ at $\chi_C$. Define
\begin{multline}\label{chiDouble}
\Gamma^{(N,M)}_{C,D}=\frac{\Gamma^{(N,M)}[e^{Ct}, e^{Dt}]}{CD(1+\Q^{(N,M)}[e^{Ct}])(1+\Q^{(N,M)}[e^{Dt}])}\\
=\frac{\calQ^{(N,M)}[(1+Ct)(1+Dt)e^{(C+D)t}]-\calQ^{(N,M)}[(1+Ct)e^{Ct}]\calQ^{(N,M)}[(1+Ct)e^{Ct}]}{CD(1+\Q^{(N,M)}[e^{Ct}])(1+\Q^{(N,M)}[e^{Dt}])},
\end{multline}
and for $v\in\calP^{(N,M)}$
\begin{equation}\label{chiSingle}
\Gamma^{(N,M)}_C(v)=\frac{\Gamma^{(N,M)}[e^{Ct}](v)}{C(1+\Q^{(N,M)}[e^{Ct}])}.
\end{equation}
Informally, we think of $\Gamma_{C,D}$ and $\Gamma_C(v)$ as $\Gamma(\chi_C, \chi_D)$ and $\Gamma(\chi_C,v)$. To be more precise:

\begin{cor}
\label{cor111}
For any $u,v\in\calP^{(N,M)}$, $C,D\in\R$ and polynomials $p,q\in\R[x]$ we have $p(\chi^{(N,M)}_C)u,q(\chi^{(N,M)}_D)v\in \D[\calE_{z,z';\theta}]$ and 
\begin{multline}
\label{Echiexpressiongen}
\calE_{z,z';\theta}(p(\chi^{(N,M)}_C)u, q(\chi^{(N,M)}_D)v)=\Big\langle  p\left(\chi^{(N,M)}_C\right)q\left(\chi^{(N,M)}_D\right)\Gamma^{(N,M)}(u,v)\\ 
+p'\left(\chi^{(N,M)}_C\right)q\left(\chi^{(N,M)}_D\right)u\Gamma^{(N,M)}_C(v)+p\left(\chi^{(N,M)}_C\right)q'\left(\chi^{(N,M)}_D\right)v\Gamma^{(N,M)}_D(u)\\
+p'\left(\chi^{(N,M)}_C\right)q'\left(\chi^{(N,M)}_D\right)uv\Gamma^{(N,M)}_{C,D}\Big\rangle_{\M}.
\end{multline}
\end{cor}
\begin{proof}
Follows directly from Proposition \ref{prop111} by setting $f(x)=p(C^{-1}\ln(1+x)), g(y)=q(D^{-1}\ln(1+y)),\phi(t)=e^{Ct}, \psi(t)=e^{Dt}$ and applying chain rule.
\end{proof}

\subsection{Asymptotics of $\Gamma^{(N,M)}_{C,D}$ and $\Gamma^{(N,M)}_C(v)$.} Our next goal is to analyze the expression from Corollary \ref{cor111} as $|C|,|D|\to\infty$. To do it we use dominated convergence, so in this subsection we get uniform bounds and pointwise limits for $\Gamma^{(N,M)}_{C,D}$ and $\Gamma^{(N,M)}_C(v)$ defined in \eqref{chiDouble}, \eqref{chiSingle}. For the upcoming computations note that
\begin{equation}
\label{bbbbb}
\Gamma_C^{(N,M)}(x_i)=\frac{-x_i\Q^{(N,M)}[(tC+1)e^{Ct}-1]}{C(1+\Q^{(N,M)}[e^{Ct}])},
\end{equation}
\begin{equation}
\label{ccccc}
\Gamma_C^{(N,M)}(q^{(N,M)}_k)=(k+1)\frac{\Q^{(N,M)}[(Ct+1)e^{Ct}t^k]-\Q^{(N,M)}[(Ct+1)e^{Ct}]q^{(N,M)}_k}{C(1+\Q^{(N,M)}[e^{Ct}])}.
\end{equation}

\begin{prop}[uniform bound]\label{lem353535}
1) For sufficiently large $C,D>0$ or sufficiently small $C,D<0$ the functions $\Gamma^{(N,M)}_{C,D}$ are uniformly bounded by a constant which is independent of $C, D$.

2) Fix $v\in\calP^{(N,M)}$. The functions $\Gamma^{(N,M)}_{C}(v)$ are uniformly bounded by a constant independent of $C$ for sufficiently large $|C|$.
\end{prop}
\begin{proof}
Let $\omega=(\alpha,\beta)\in\Omega$ denote an arbitrary point. We first deal with $\Gamma^{(N,M)}_{C,D}$.  Note that applying Lemma \ref{expineq} to $\calQ^{(N,M)}[t^ke^{Ct}](\omega)=\calQ[t^ke^{Ct}](\tau_\alpha^N\tau_\beta^M\omega)$ we get for $C>0$
\begin{equation}\label{ineqCpos}
|\calQ^{(N,M)}[t^ke^{Ct}](\omega)|\leq \max(\alpha_{N+1}^ke^{\alpha_{N+1}C};\theta^k)\leq \alpha_{N+1}^ke^{\alpha_{N+1}C}+\theta^k,
\end{equation}
$$
\calQ^{(N,M)}[e^{Ct}](\omega)\geq \alpha_{N+1}e^{\alpha_{N+1}C}
$$
and for $C<0$
\begin{equation}\label{ineqCneg}
|\calQ^{(N,M)}[t^ke^{Ct}](\omega)|\leq \max(\theta^k\beta_{M+1}^ke^{-\theta\beta_{M+1}C};1)\leq \theta^k\beta_{M+1}^ke^{-\theta\beta_{M+1}C}+1,
\end{equation}
$$
\calQ^{(N,M)}[e^{Ct}](\omega)\geq \beta_{M+1}e^{-\theta\beta_{M+1}C}.
$$
Applying these inequalities to \eqref{chiDouble} we get for $C,D>0$
\begin{multline}\label{shiftedQinqPos}
|\Gamma_{C,D}^{(N,M)}(\omega)|\leq\frac{(C\alpha_{N+1}+1)(D\alpha_{N+1}+1)e^{(C+D)\alpha_{N+1}}+(C\theta +1)(D\theta +1)}{CD(1+\alpha_{N+1} e^{C\alpha_{N+1}})(1+\alpha_{N+1} e^{D\alpha_{N+1}})}\\
+\frac{((C\alpha_{N+1}+1)e^{C\alpha_{N+1}}+C\theta +1)((D\alpha_{N+1}+1)e^{D\alpha_{N+1}}+D\theta+1)}{CD(1+\alpha_{N+1} e^{C\alpha_{N+1}})(1+\alpha_{N+1} e^{D\alpha_{N+1}})},
\end{multline}
and for $C,D<0$
\begin{multline}\label{shiftedQinqNeg}
|\Gamma_{C,D}^{(N,M)}(\omega)|\leq\frac{(|C|\theta\beta_{M+1}+1)(|D|\theta\beta_{M+1}+1)e^{\theta\beta_{M+1}(|C|+|D|)}+(|C| +1)(|D| +1)}{|C||D|(1+\beta_{M+1} e^{|C|\theta\beta_{M+1}})(1+\beta_{M+1} e^{|D|\theta\beta_{M+1}})}\\
+\frac{((|C|\beta_{M+1}+1)e^{\theta\beta_{M+1}|C|}+|C| +1)((|D|\theta\beta_{M+1}+1)e^{\theta\beta_{M+1}|D|}+|D|+1)}{|C||D|(1+\beta_{M+1} e^{|C|\theta\beta_{M+1}})(1+\beta_{M+1} e^{|D|\theta\beta_{M+1}})}.
\end{multline}

Note that for $|C|>1$ 
\be
\frac{|C|\theta +1}{|C|} < \theta +1, \quad \frac{|C| +1}{|C|} < 2,
\ee
and so it is enough to find a uniform bound for expressions of the form
\be
\frac{(|C|\lambda x+1)e^{|C|\lambda x}}{|C|(1+x e^{|C|\lambda x})}
\ee
for varying $x\in[0,1]$ and  fixed $\lambda>0$ (to get bounds on \eqref{shiftedQinqPos}, \eqref{shiftedQinqNeg} set $x=\alpha_{N+1}$, $\lambda=1$ or $x=\beta_{M+1}, \lambda=\theta$). 

To be precise, it is enough to prove that
\begin{equation}
\label{ineq}
\sup_{\substack{x\in[0;1] \\ C>2\lambda^{-1}}} \frac{(C\lambda x+1)e^{C\lambda x}}{C(1+x e^{C\lambda x})}<\infty.
\end{equation}
Note that
\be
\frac{(C\lambda x+1)e^{C\lambda x}}{C(1+x e^{C\lambda x})}=\frac{\lambda x+C^{-1}}{x+e^{-\lambda x C}}.
\ee
For $x \leq (C\lambda)^{-1}\ln (C\lambda)$ we have
\be
\frac{\lambda x+C^{-1}}{x+e^{-\lambda xC}}\leq \frac{\lambda x+C^{-1}}{x+e^{-\ln (C\lambda)}}=\lambda,
\ee
while otherwise $x > (C\lambda)^{-1}\ln (C\lambda)$ and
\be
\frac{\lambda x+C^{-1}}{x+e^{-\lambda xC}}\leq \frac{\lambda x+C^{-1}}{x}=\lambda+C^{-1}x^{-1}\leq \lambda+\frac{\lambda}{\ln (\lambda C)}\leq \lambda+\frac{\lambda}{\ln 2}.
\ee

Now consider $\Gamma_C^{(N,M)}(v)$. Since it is a differential operator with respect to $v$ we can use the product rule, so it is enouigh to consider the cases when $v$ is a generator of $\calP^{(N,M)}$. 

In the case $v=x_i$ for the right-hand side of \eqref{bbbbb} inequalities \eqref{ineqCpos}, \eqref{ineqCneg} give the bound
\be
\left|\frac{-x_i\Q^{(N,M)}[(tC+1)e^{Ct}](\omega)+x_i}{C(1+\Q^{(N,M)}[e^{Ct}](\omega))}\right|\leq\frac{x_i((|C|\lambda y+1)e^{|C|\lambda y}+|C|\lambda^{-1}\theta+1)+x_i}{|C|(1+y e^{|C|\lambda y})}
\ee
where $\lambda=1, y=\alpha_{N+1}$ if $C>0$ and $\lambda=\theta, y=\beta_{M+1}$ if $C<0$. One can readily see that this expression is uniformly bounded using \eqref{ineq} again.

For $v=q_k^{(N,M)}$ we split \eqref{ccccc} into two terms:
\be
\Gamma_C^{(N,M)}(q_k)=(k+1)\frac{\Q^{(N,M)}[(Ct+1)e^{Ct}t^k]}{C(1+\Q^{(N,M)}[e^{Ct}])}-(k+1)\frac{\Q^{(N,M)}[(Ct+1)e^{Ct}]q_k}{C(1+\Q^{(N,M)}[e^{Ct}])}
\ee
The second term is bounded by \eqref{ineq} in the same way as in the previous case. For the first term Lemma \ref{expineq} implies for $C>0$
\be
\left|\frac{\Q^{(N,M)}[(Ct+1)e^{Ct}t^k](\omega)}{C(1+\Q^{(N,M)}[e^{Ct}](\omega))}\right|\leq \frac{(C\alpha_{N+1}+1)\alpha_{N+1}^ke^{C\alpha_{N+1}}+(C\theta +1)\theta^k}{C(1+\alpha_{N+1} e^{C\alpha_{N+1}})}
\ee
and for $C<0$
\be
\left|\frac{\Q^{(N,M)}[(Ct+1)e^{Ct}t^k](\omega)}{C(1+\Q^{(N,M)}[e^{Ct}](\omega))}\right|\leq \frac{(|C|\theta\beta_{M+1}+1)\beta_{M+1}^ke^{|C|\theta\beta_{M+1}}+(|C|+1)}{C(1+\beta_{M+1} e^{|C|\theta\beta_{N+1}})}.
\ee
Both expressions can be readily bounded using \eqref{ineq}.
\end{proof}

The pointwise limit below is the reason why we need Conjecture \ref{simple-assumption}.

\begin{prop}[pointwise convergence] 
\label{lem444} Assuming that Conjecture \ref{simple-assumption} holds:

1) Functions $\Gamma^{(N,M)}_{C,D}$ converge point-wise $\M$-a.e. to 
\begin{align*}
&\Gamma^{(N+1,M)}(\alpha_{N+1},\alpha_{N+1})=\alpha_{N+1}-\alpha_{N+1}^2\qquad &\text{as}\quad C,D&\to\infty,\\
&\Gamma^{(N,M+1)}(-\theta\beta_{M+1},-\theta\beta_{M+1})=\theta^2(\beta_{M+1}-\beta_{M+1}^2) &\text{as}\quad C,D&\to-\infty.
\end{align*}

2) For fixed $v\in\calP^{(N,M)}$ the functions $\Gamma_C^{(N,M)}(v)$ converge pointwise $\M$-a.e. to 
\begin{equation}\label{lem444p2e1}
\Gamma^{(N+1,M)}(\alpha_{N+1},v)=\sum_{i=-M}^N(-\alpha_{N+1}x_i)\frac{\partial v}{\partial x_i}+\sum_{k\geq1}(k+1)(\alpha_{N+1}^{k+1}-\alpha_{N+1}q^{(N,M)}_k)\frac{\partial v}{\partial q^{(N,M)}_k}
\end{equation}
as $C\to\infty$ and to
\begin{equation}\label{lem444p2e2}
\Gamma^{(N,M+1)}(-\theta\beta_{M+1},v)=\sum_{i=-M}^N(\theta\beta_{M+1}x_i)\frac{\partial v}{\partial x_i}+\sum_{k\geq1} (-\theta)(k+1)((-\theta)^k\beta_{M+1}^{k+1}-\beta_{M+1}q^{(N,M)}_k)\frac{\partial v}{\partial q^{(N,M)}_k}
\end{equation}
as $C\to-\infty$.
\end{prop}
\begin{proof}
First note that by Conjecture \ref{simple-assumption} the set $\{(\alpha,\beta)\in\Omega | \alpha_{N+1}=\alpha_{N+2}>0\}$ has measure $0$, hence to prove $\M$-a.e. convergence it is enough to consider two cases: either $\alpha_{N+1}>\alpha_{N+2}$ or $\alpha_{N+1}=0$. Similarly we can assume that either $\beta_{M+1}>\beta_{M+2}$ or $\beta_{M+1}=0$. So, we fix a point $\omega=(\alpha,\beta)$ satisfying these conditions. 

Our main tool is Lemma \ref{exasympt} applied to the point $\tau_\alpha^N\tau_\beta^M\omega$. Namely for any $k\in\mathbb Z_{\geq 0}$ as $C\to\infty$ we have
$$
\Q^{(N,M)}[(Ct+1)t^ke^{tC}](\omega)=\nu_{\tau^{N}_\alpha\tau_\beta^M\omega}(\{\alpha_{N+1}\})(C\alpha_{N+1}+1)\alpha_{N+1}^ke^{\alpha_{N+1}C}+o(Ce^{\alpha_{N+1}C}),
$$
where we use little-$o$ notation. Note that when $\alpha_{N+1}>\alpha_{N+2}$ we have $\nu_{\tau^{N}_\alpha\tau_\beta^M\omega}(\{\alpha_{N+1}\})=\alpha_{N+1}$, while for $\alpha_{N+1}=0$ we have 
$$
\nu_{\tau^{N}_\alpha\tau_\beta^M\omega}(\{\alpha_{N+1}\})(C\alpha_{N+1}+1)\alpha_{N+1}^ke^{\alpha_{N+1}C}=\nu_{\tau^{N}_\alpha\tau_\beta^M\omega}(\{0\})\alpha_{N+1}^k.
$$
In either case we get
\begin{equation}\label{singleCpospw}
\Q^{(N,M)}[(Ct+1)t^ke^{tC}](\omega)=\alpha^{k+2}_{N+1}Ce^{\alpha_{N+1}C}+o(Ce^{\alpha_{N+1}C}),\qquad C\to\infty.
\end{equation}
For $C\to-\infty$ one can do identical computations to obtain
\begin{equation}\label{singleCnegpw}
\Q^{(N,M)}[(Ct+1)t^ke^{tC}](\omega)=\beta_{M+1}(-\theta\beta_{M+1})^{k+1}Ce^{-\theta\beta_{M+1}C}+o(Ce^{-\theta\beta_{M+1}C}),\qquad C\to-\infty.
\end{equation}
For the computation below we also often use the relations like
$$
\frac{\lambda^{2}}{e^{-\lambda C}+\lambda+o(1)}=\lambda+o(1), \qquad C\to\infty,
$$
for $\lambda\in\mathbb R_{\geq 0}$, which can be checked immediately by considering cases $\lambda\neq 0$ and $\lambda=0$ separately. 

1) 
First take $C,D\to\infty$. Applying Lemma \ref{exasympt} to the point $\tau_\alpha^N\tau_\beta^M\omega$ we have
\begin{multline*}
\Q^{(N,M)}[(Ct+1)(Dt+1)e^{t(C+D)}](\omega)\\
=\nu_{\tau^{N}_\alpha\tau_\beta^M\omega}(\{\alpha_{N+1}\})(C\alpha_{N+1}+1)(D\alpha_{N+1}+1)e^{\alpha_{N+1}(C+D)}+o(CDe^{\alpha_{N+1}(C+D)})\\
=CD\alpha^3_{N+1}e^{\alpha_{N+1}(C+D)}+o(CDe^{\alpha_{N+1}(C+D)}).
\end{multline*}
where we use the same argument as in \eqref{singleCpospw}. Then \eqref{chiDouble} combined with \eqref{singleCpospw} results in
\be
\Gamma_{C,D}^{(N,M)}(\omega)=\frac{CD(\alpha_{N+1}-\alpha_{N+1}^2)\alpha_{N+1}^2e^{\alpha_{N+1}(C+D)} + o(CDe^{\alpha_{N+1}(C+D)})}{CD(1+\alpha_{N+1}e^{\alpha_{N+1}C})(1+\alpha_{N+1}e^{\alpha_{N+1}D})+o(CDe^{\alpha_{N+1}(C+D)})}
=\alpha_{N+1}-\alpha_{N+1}^2 +o(1).
\ee

The computation for $C,D\to-\infty$ are similar: by Lemma \ref{exasympt} we have
\begin{multline*}
\Q^{(N,M)}[(Ct+1)(Dt+1)e^{t(C+D)}](\omega)\\
=\nu_{\tau^{N}_\alpha\tau_\beta^M\omega}(\{-\theta\beta_{M+1}\})(-C\theta \beta_{M+1}+1)(-D\theta\beta_{M+1}+1)e^{-\theta\beta_{M+1}(C+D)}+o(CDe^{-\theta\beta_{M+1}(C+D)})\\
=CD\theta^2\beta^3_{M+1}e^{-\theta\beta_{M+1}(C+D)}+o(CDe^{-\theta\beta_{M+1}(C+D)})
\end{multline*}
and \eqref{chiDouble},\eqref{singleCnegpw} lead to
\begin{multline*}
\Gamma_{C,D}^{(N,M)}(\omega)=\frac{(\beta_{M+1}-\beta_{M+1}^2)CD(-\theta\beta_{M+1})^2e^{-\theta\beta_{M+1}(C+D)} + o(CDe^{-\theta\beta_{M+1}(C+D)})}{CD(1+\beta_{M+1}e^{-\theta\beta_{M+1}C})(1+\beta_{M+1}e^{-\theta\beta_{M+1}D})+o(CDe^{-\theta\beta_{N+1}(C+D)})}\\
=\theta^2(\beta_{M+1}-\beta_{M+1}^2) +o(1).
\end{multline*}

2) For this part we first explain equations \eqref{lem444p2e1} and \eqref{lem444p2e2}. Note that both $\Gamma^{(N+1,M)}(\alpha_{N+1}, v)$ and $\Gamma^{(N+1,M)}(-\theta\beta_{N+1}, v)$ are differential operators in $v\in\calP^{(N,M)}$, so it is enough to check the equations when $v=x_i$ and $v={q_{k}^{(N,M)}}$. The former case follows immediately from \eqref{Gammaval}, while the latter case follows from the following computation:
\begin{multline*}
\Gamma^{(N+1,M)}(\alpha_{N+1}, q_k^{(N,M)})=\Gamma^{(N+1,M)}(\alpha_{N+1}, \alpha_{N+1}^{k+1}+q_k^{(N+1,M)})\\
=(k+1)\alpha_{N+1}^{k+1}-(k+1)\alpha_{N+1}^{k+2}-(k+1)\alpha_{N+1}q_k^{(N+1,M)}=(k+1)\left(\alpha_{N+1}^{k+1}-\alpha_{N+1}q_k^{(N,M)}\right).
\end{multline*}
\begin{multline*}
\Gamma^{(N,M+1)}(\beta_{M+1}, q_k^{(N,M)})=\Gamma^{(N,M+1)}(\beta_{M+1}, (-\theta)^k\beta_{M+1}^{k+1}+q_k^{(N+1,M)})\\
=(k+1)(-\theta)^k\beta_{M+1}^{k+1}-(k+1)(-\theta)^k\beta_{M+1}^{k+2}-(k+1)\beta_{M+1}q_k^{(N,M+1)}=(k+1)\left((-\theta)^k\beta_{M+1}^{k+1}-\beta_{M+1}q_k^{(N,M)}\right).
\end{multline*}

To show the pointwise convergence note that $\Gamma_C^{(N,M)}(v)$ is also a differential operator in $v$, so it is again enough to consider $v=x_i$ and $v=q^{(N,M)}_k$. For $v=x_i$ we apply \eqref{singleCpospw},\eqref{singleCnegpw} to \eqref{bbbbb} getting
\be
\left[\Gamma_C^{(N,M)}(x_i)\right](\omega)=\frac{-x_iC\alpha_{N+1}^2e^{C\alpha_{N+1}}+o(Ce^{C\alpha_{N+1}})}{C(1+\alpha_{N+1}e^{C\alpha_{N+1}})+o(Ce^{C\alpha_{N+1}})}=-x_i\alpha_{N+1}+o(1)\qquad C\to\infty,
\ee
\be
\left[\Gamma_C^{(N,M)}(x_i)\right](\omega)=\frac{-x_iC(-\theta\beta_{M+1})\beta_{M+1}e^{-\theta C\beta_{M+1}}+o(Ce^{-\theta C\beta_{N+1}})}{C(1+\beta_{M+1}e^{-\theta C\beta_{M+1}})+o(Ce^{-\theta C\beta_{M+1}})}=x_i\theta\beta_{M+1}+o(1) \qquad C\to-\infty.
\ee
In the case $v=q_k^{(N,M)}$ we use \eqref{ccccc} to get for $C\to\infty$
\begin{multline*}
\left[\Gamma_C^{(N,M)}\left(q^{(N,M)}_k\right)\right](\omega)=(k+1)\frac{C(\alpha_{N+1}^{k+2}-q^{(N,M)}_k\alpha_{N+1}^2)e^{C\alpha_{N+1}} +o(Ce^{C\alpha_{N+1}})}{C(1+\alpha_{N+1}e^{C\alpha_{N+1}})+o(Ce^{C\alpha_{N+1}})}\\
=(k+1)\left(\alpha_{N+1}^{k+1}-\alpha_{N+1}q^{(N,M)}_k\right)+o(1).
\end{multline*}
while for $C\to-\infty$
\begin{multline*}
\left[\Gamma_C^{(N,M)}\left(q^{(N,M)}_k\right)\right](\omega)=(k+1)\frac{C((-\theta)^{k+1}\beta_{M+1}^{k+2}-q^{(N,M)}_k(-\theta)\beta_{M+1}^2)e^{-\theta C\beta_{M+1}}+o(Ce^{-\theta C\beta_{M+1}}) }{C(1+\beta_{M+1}e^{-\theta C\beta_{M+1}})+o(Ce^{-\theta C\beta_{M+1}})}\\
=-\theta(k+1)\left((-\theta)^{k}\beta_{M+1}^{k+1}-\beta_{M+1}q^{(N,M)}_k\right)+o(1).
\end{multline*}
\end{proof}

\subsection{Proof of Theorem \ref{Evalue}.} Now we collect the results above to prove the inductive step in Theorem \ref{Evalue}.

Recall that our goal is to show that the restriction of $\calE_{z,z';\theta}$ to $\calP^{(N+1,M)}$ coincides with $E^{(N+1,M)}_{z,z';\theta}$ defined by \eqref{Evaluedef} and similarly for $(N,M+1)$. Note that we have
\begin{equation}
\label{qchange}
q_k^{(N,M)}=q_k^{(N+1,M)} + \alpha_{N+1}^{k+1},\quad q_k^{(N,M)}=q_k^{(N,M+1)} + (-\theta)^k\beta_{M+1}^{k+1}
\end{equation}
so we have natural embeddings $\calP^{(N,M)}\subset \calP^{(N+1,M)}$ and $\calP^{(N,M)}\subset \calP^{(N,M+1)}$. Moreover, from \eqref{qchange} we have $\calP^{(N+1,M)}=\calP^{(N,M)}[\alpha_{N+1}]$ and $\calP^{(N,M+1)}=\calP^{(N,M)}[\beta_{M+1}]$. Hence, by linearity of the form $\calE_{z,z';\theta}$, it is enough to show that
\begin{equation}\label{aStep}
\calE_{z,z';\theta}(p(\alpha_{N+1})u,q(\alpha_{N+1})v)=\lan\Gamma^{(N+1,M)}(p(\alpha_{N+1})u,q(\alpha_{N+1})v)\ran_{\M}.
\end{equation}
\begin{equation}\label{bStep}
\calE_{z,z';\theta}(p(\beta_{M+1})u,q(\beta_{M+1})v)=\lan\Gamma^{(N,M+1)}(p(\beta_{M+1})u,q(\beta_{M+1})v)\ran_{\M}.
\end{equation}
for polynomials $p(t),q(t)\in\mathbb R[t]$ and $u,v\in\calP^{(N,M)}$.

For now let us focus on \eqref{aStep}. By the differentiation rules the integrand in the right-hand side of \eqref{aStep} is given by 
\begin{multline}
\label{EStepexpressiongen}
\Gamma^{(N+1,M)}(p(\alpha_{N+1})u, q(\alpha_{N+1})v)=p(\alpha_{N+1})q(\alpha_{N+1})\Gamma^{(N+1,M)}(u,v) \\
+p'(\alpha_{N+1})q(\alpha_{N+1})u\Gamma^{(N+1,M)}(\alpha_{N+1},v)+p(\alpha_{N+1})q'(\alpha_{N+1})v\Gamma^{(N+1,M)}(\alpha_{N+1},u)\\
+ p'(\alpha_{N+1})q'(\alpha_{N+1})uv\Gamma^{(N+1,M)}(\alpha_{N+1},\alpha_{N+1}).
\end{multline}

On the other hand, recall that by Propositions \ref{chiBound}, \ref{chiLimit} the functions $\chi_C^{(N,M)}$ are uniformly bounded and converge pointwise to $\alpha_{N+1}$ as $C\to\infty$. Hence, by dominated convergence, $p\left(\chi_C^{(N,M)}\right)u$ converges to $p(\alpha_{N+1})u$ in $L^2(\Omega,\mu)$. Corollary \ref{cor111} tells that $p\left(\chi_C^{(N,M)}\right)u\in\D[\calE_{z,z';\theta}]$ and the value of $\calE_{z,z';\theta}$ on such functions is given by \eqref{Echiexpressiongen}.
By Propositions \ref{lem353535} and \ref{lem444} the expressions $\Gamma^{(N,M)}_C(v)$ and $\Gamma^{(N,M)}_{C,D}$ appearing in \eqref{Echiexpressiongen} are uniformly bounded by a constant and converge $\M$-a.e. pointwise to $\Gamma^{(N+1,M)}(\alpha_{N+1},v)$ and $\Gamma^{(N+1,M)}(\alpha_{N+1},\alpha_{N+1})$ respectively. Combined with the convergence of $\chi_C^{(N,M)}$ we see that \eqref{Echiexpressiongen} gives
\begin{multline}\label{limitEexpression}
\lim_{C,D\to\infty}\calE_{z,z';\theta}(p(\chi^{(N,M)}_C)u, q(\chi^{(N,M)}_D)v)=\Big\langle  p\left(\alpha_{N+1}\right)q\left(\alpha_{N+1}\right)\Gamma^{(N,M)}(u,v)\\ 
+p'\left(\alpha_{N+1}\right)q\left(\alpha_{N+1}\right)u\Gamma^{(N+1,M)}(\alpha_{N+1},v)+p\left(\alpha_{N+1}\right)q'\left(\alpha_{N+1}\right)v\Gamma^{(N+1,M)}(u,\alpha_{N+1})\\
+p'\left(\alpha_{N+1}\right)q'\left(\alpha_{N+1}\right)uv\Gamma^{(N+1,M)}(\alpha_{N+1},\alpha_{N+1})\Big\rangle_{\M}.
\end{multline}
Then by Proposition \ref{Eprop} we get $p(\alpha_{N+1})u,q(\alpha_{N+1})v\in\D[\calE_{z,z';\theta}]$ and $\calE_{z,z';\theta}(p(\alpha_{N+1})u,q(\alpha_{N+1})v)$ is given by the right-hand side of \eqref{limitEexpression}.  

To finish the proof of \eqref{aStep} we just need to compare the right-hand sides of \eqref{EStepexpressiongen} and \eqref{limitEexpression}. One can easily see that the only difference is in the first term, and it is covered by the statement below
\begin{lem}\label{consistentlem} For any $u,v\in\calP^{(N,M)}$ we have $\Gamma^{(N,M)}(u,v)=\Gamma^{(N+1,M)}(u,v)$
\end{lem}
\begin{proof}
Both sides are bi-differential operators in $u,v$, so it is enough to consider the cases when $u,v$ are generators of $\calP^{(N.M)}$. These cases can be directly verified: for $k,l\in\mathbb Z_{\geq 1}$, $i,j=-M,-M+1, \dots, N$ we have
$$
\Gamma^{(N+1,M)}(x_i, x_j)=\1_{i=j} x_i-x_ix_j=\Gamma^{(N,M)}(x_i, x_j),
$$
\begin{multline*}
\Gamma^{(N+1,M)}(x_i, q^{(N,M)}_k)=\Gamma^{(N+1,M)}(x_i, \alpha_{N+1}^{k+1}+q^{(N+1,M)}_k)\\
=-(k+1)x_i\alpha_{N+1}^{k+1}-(k+1)x_iq^{(N+1,M)}_k=-(k+1)x_iq^{(N,M)}_k=\Gamma^{(N,M)}(x_i, q^{N,M}_k),
\end{multline*}
\begin{multline*}
\Gamma^{(N+1,M)}(q^{(N,M)}_k, q^{(N,M)}_l)=\Gamma^{(N+1,M)}(\alpha_{N+1}^{k+1}+q^{(N+1,M)}_k, \alpha_{N+1}^{l+1}+q^{(N+1,M)}_l)\\
=(k+1)(l+1)\left[\alpha_{N+1}^{k+l+1}-\alpha_{N+1}^{k+l+2}-\alpha_{N+1}^{k+1}q^{(N+1,M)}_l-\alpha_{N+1}^{l+1}q^{(N+1,M)}_k+q_{k+l}^{(N+1,M)}-q_{k}^{(N+1,M)} q_{l}^{(N+1,M)}  \right]\\
=(k+1)(l+1)\left[q_{k+l}^{(N,M)}-q_{k}^{(N,M)} q_{l}^{(N,M)}  \right]=\Gamma^{(N,M)}(q^{(N,M)}_k, q^{(N,M)}_l).
\end{multline*}
\end{proof}

The proof of the other half  \eqref{bStep} is identical to the argument above so we only briefly outline the needed computations. Recall that $\chi^{(N,M)}_C$ as $C\to-\infty$ are uniformly bounded and converge pointwise to $-\theta\beta_{M+1}$. Hence to get $\calE_{z,z';\theta}(p(\beta_{M+1})u,q(\beta_{M+1})v)$ we need to consider the limits 
$$\lim_{C,D\to-\infty}\calE_{z,z';\theta}\left(p\left(-\theta^{-1}\chi_{C}^{(N,M)}\right)u,q\left(-\theta^{-1}\chi_{D}^{(N,M)}\right)v\right).$$
Such limits can again be computed using Corollary \ref{cor111} and Propositions \ref{lem353535} and \ref{lem444}, getting
\begin{multline*}
\lim_{C,D\to-\infty}\calE(p(-\theta^{-1}\chi^{(N,M)}_C)u, q(-\theta^{-1}\chi^{(N,M)}_D)v)=\Big\langle  p\left(\beta_{M+1}\right)q\left(\beta_{M+1}\right)\Gamma^{(N,M)}(u,v)\\ 
+(-\theta^{-1})p'\left(\beta_{M+1}\right)q\left(\beta_{M+1}\right)u\Gamma^{(N,M+1)}(-\theta\beta_{M+1},v)+p\left(\beta_{M+1}\right)(-\theta^{-1})q'\left(\beta_{M+1}\right)v\Gamma^{(N,M+1)}(u,-\theta\beta_{M+1})\\
+(-\theta^{-1})p'\left(\beta_{M+1}\right)(-\theta^{-1})q'\left(\beta_{M+1}\right)uv\Gamma^{(N,M+1)}(-\theta\beta_{M+1},-\theta\beta_{M+1})\Big\rangle_{\M}.
\end{multline*}
By direct computations similar to Lemma \ref{consistentlem} one can show that $\Gamma^{(N,M)}(u,v)=\Gamma^{(N,M+1)}(u,v)$ and then \eqref{bStep} follows since the integrand in the right-hand side above is exactly $\Gamma^{(N,M+1)}(p(\beta_{M+1})u,q(\beta_{M+1})v)$. 
 
So we have verified the induction steps $(N,M)\implies (N+1,M)$ and $(N,M)\implies (N,M+1)$, finishing the proof of Theorem \ref{Evalue}. \qed

\end{section}

\begin{section}{Process $X_{z,z';\theta}$ never leaves $\Omega_0$}
In this section we show that the process $X_{z,z';\theta}(t)$ almost surely for all $t>0$ exists on the face $\Omega_0$, under the assumption that Conjectures \ref{boundary-assumption} and \ref{simple-assumption} hold. In particular, for $\theta=1$ the assumption is true and so our result unconditionally holds in this case.

We start by considering a special case when the process starts from its symmetrizing distribution $\M$.
\begin{prop}\label{boundaryfirststep}
Assume that $\theta=1$ or $\theta$ takes any other positive value for which Conjectures \ref{boundary-assumption}, \ref{simple-assumption} hold. Then we have 
\be
\P_{\M}(X_{z,z';\theta}(t)\in\Omega_0, \forall t\geq0):=\int_{\Omega} \P_{\omega}(X(t)\in\Omega_0, \forall t\geq0)\M(d\omega)=1.
\ee
\end{prop}
\begin{proof}
Our argument is similar to \cite[Proposition 2.2]{FS09}. For the duration of the proof almost sure statements are considered with respect to $\P_{\M}$.

Consider functions $\varphi_n=\sum_{i=1}^n\alpha_i+\sum_{j=1}^n\beta_j$ and $\varphi=\sum_{i\geq1}\alpha_i+\sum_{j\geq1}\beta_j$. By dominated convergence $\varphi_n$ converge to $\varphi$ in $L^2(\Omega,\M)$. At the same time, by Theorem \ref{result} for $n>m$ the following holds
\begin{multline*}
\mathcal E(\varphi_n-\varphi_m,\varphi_n-\varphi_m)=\int_{\Omega}\left(\,\sum_{i=n+1}^m\, \alpha_i+\,\sum_{j=n+1}^m\, \beta_i-\left(\,\sum_{i=n+1}^m\,\alpha_i+\,\sum_{j=n+1}^m\,\beta_j\right)^2\right)d\M\\\leq
\int_{\Omega}\left(\sum_{i\geq n+1}\alpha_i+\sum_{j\geq n+1} \beta_j\right)d\M
\end{multline*} 

The integral in the last expression converges to $0$ as $n\to\infty$ by dominated convergence because the functions $\sum_{i\geq n+1}\alpha_i+\sum_{j\geq n+1} \beta_j$ are uniformly bounded by $1$ and converge pointwise to $0$. So the sequence $(\varphi_n)$ is $\calE_1$-Cauchy and we can apply Lemma \ref{uniformLem} to get for any $T\in\mathbb R_{>0}$  subsequence $n_k$ such that
\be
\mathbb P_{\M}(\varphi_{n_k}(X_{z,z';\theta}(t))\ \text{converge\ uniformly\ on}\ [0;T])=1.
\ee
Since $\varphi_n$ converge pointwise to $\varphi$, we know that the uniform limit above must be $\varphi(X_{z,z';\theta}(t))$. Since $\varphi_{n_k}\in C(\Omega)$ and $X_{z,z';\theta}$ has continuous trajectories a.s., the functions $\varphi_{n_k}(X_{z,z';\theta}(t))$ are continuous with respect to $t$ a.s., hence $\varphi(X_{z,z';\theta}(t))$ is a uniform limit of continuous functions a.s. and we get
$$\mathbb P_{\M}(\varphi(X(t))\ \text{is\ continuous\ in}\ t\in\mathbb R_{\geq 0})=1.$$

By Conjecture \ref{boundary-assumption} for any fixed $t\geq 0$ we have $\P_{\M}(X_{z,z';\theta}(t)\in\Omega_0)=\M(\Omega_0)=1$ since $\M$ is the invariant measure of $X_{z,z';\theta}$. Then for any fixed $t\in\mathbb R_{\geq 0}$ we have $\varphi(X(t))=1$ a. s. and since $\varphi(X(t))$ is continuous a.s. we get 
$$\P_{\M}(X_{z,z';\theta}(t)\in\Omega_0,  \forall t\geq0)=\mathbb P_{\M}(\varphi(X(t))=1,\forall  t\geq0)=1.$$
\end{proof}

Now we consider the general statement
\begin{theo} \label{boundary-result} Assume that $\theta=1$ or $\theta$ takes any other positive value for which Conjectures \ref{boundary-assumption}, \ref{simple-assumption} hold. Then for any $\omega\in\Omega$ we have
\be
\mathbb P_{\omega}(X_{z,z';\theta}(t)\in\Omega_0, \forall t>0)=1.
\ee
\end{theo}
\begin{proof}
The proof repeats the argument form \cite{Eth14}, for completeness sake we present it here. For any $\omega\in\Omega$ and $s>0$ by Markov property we have
\begin{multline*}
\mathbb P_{\omega}(X_{z,z';\theta}(t)\in\Omega_0, \forall t\geq s)=\E_\omega\left[\P_{\omega}(X_{z,z';\theta}(t)\in\Omega_0, \forall t\geq s \mid X_{z,z';\theta}(r), 0\leq r\leq s)\right]\\
=\E_\omega[\P_{X_{z,z';\theta}(s)}(X_{z,z';\theta}(t)\in\Omega_0, \forall t\geq 0)]=\int_{\Omega}\P_{\widetilde{\omega}}(X_{z,z';\theta}(t)\in\Omega_0, \forall t\geq 0) p_s(\omega, d\widetilde{\omega}),
\end{multline*}
where $p_s(\omega, d\widetilde{\omega})$ denotes the transition function of $X_{z,z';\theta}$. By Theorem \ref{density-theorem} $p_s(\omega, d\widetilde{\omega})$ is absolutely continuous with respect to $\M$ and by Proposition \ref{boundaryfirststep} 
$$
\P_{\widetilde{\omega}}(X_{z,z';\theta}(t)\in\Omega_0, \forall t\geq 0)=1\qquad \text{a.e.-}\M(d\widetilde\omega).
$$
Hence the last identity holds a.e.-$p_s(\omega, d\widetilde\omega)$ as well and
$$
\mathbb P_{\omega}(X_{z,z';\theta}(t)\in\Omega_0, \forall t\geq s)=\int_{\Omega}\P_{\widetilde{\omega}}(X_{z,z';\theta}(t)\in\Omega_0, \forall t\geq 0) p_s(\omega, d\widetilde{\omega})=\int_{\Omega} p_s(\omega, d\widetilde{\omega})=1.
$$
Since $s>0$ is arbitrary, the claim follows.
\end{proof}
\end{section}

\section{Generator $A_{z,z';\theta}$ in natural coordinates.}\label{gen-sect}

The majority of this work is devoted to evaluating the Dirichlet form $\calE_{z,z';\theta}$ at natural coordinates. Here we discuss a similar question for the generator $A_{z,z';\theta}$, namely if the generator can be rewritten in terms of natural coordinates and if these coordinates are in $\calD[\overline{A}_{z,z';\theta}]$. We start this section by describing the case of Petrov's diffusion \cite{Pe07}, then we present an expression for $A_{z,z'\theta}$ in terms of natural coordinates and in the end discuss the possibility of $\calD[\overline{A}_{z,z';\theta}]$ containing natural coordinates.

\subsection{Petrov's diffusion}
Consider the infinite simplex corresponding to $\alpha$-coordinates:
$$
\overline{\nabla}_\infty=\{(\alpha,\beta)\in\Omega\mid \beta_1=\beta_2=\dots=0\}.
$$
Petrov's diffusion is a process on $\overline{\nabla}_\infty$ which depends on a pair of parameters $(a,\tau)$ and is defined by the generator
$$
A^{Pet}_{a,\tau}=\sum_{k,l\geq1}(k+1)(l+1)(q_{k+l}-q_kq_l)\frac{\partial^2}{\partial q_k\partial q_l}+\sum_{k\geq1}\big(-(k+1)(k+\tau)q_k+(k+1)(k-a)q_{k-1}\big)\frac{\partial}{\partial q_k}
$$
acting on $\Lambda^\circ=\mathbb R[q_1, q_2, \dots]$. The Poisson-Dirichlet measure $PD(a,\tau)$ is the unique symmetrizing measure of the resulting process. When $a=0$ the generator above degenerates to the generator of Ethier-Kurtz's process from \cite{EK81}. In the regime
$$
\theta\to 0, \qquad zz'\theta^{-1}\to\tau,\qquad z+z'\to -a
$$
the generator $A_{z,z';\theta}$ degenerates\footnote{One should be careful about extending this degeneration to the processes, since for $\tau\neq 0$ the limit regime forces one to leave the set of allowed triples $(z,z';\theta)$ for the process $X_{z,z';\theta}$.} to $A^{Pet}_{\tau,a}$.

It was shown in \cite{Pe07} that $A^{Pet}_{a,\tau}$ has an alternative form, namely
\begin{equation}
\label{petgen-nat}
A^{Pet}_{a,\tau}=\sum_{i\geq 1}\alpha_i\frac{\partial^2}{\partial \alpha_i^2}-\sum_{i,j\geq 1}\alpha_i\alpha_j\frac{\partial^2}{\partial \alpha_i\partial \alpha_j}-\sum_{i\geq 1}(\tau \alpha_i+a)\frac{\partial}{\partial \alpha_i}.
\end{equation}
However, one should be careful when using this form: $A^{Pet}_{a,\tau}$ is still an operator acting only on $\Lambda^\circ$ and the equality above only tells that the formal application of the right-hand side to $u\in\Lambda^\circ$ coincides with $A^{Pet}_{a,\tau}u$ on a dense subset of $\overline{\nabla}_\infty$. This gives rise to the following question: what functions can be found in the domain of the closure $\overline{A}^{Pet}_{a,\tau}$ in $L^2(\overline{\nabla}_\infty, PD(a,\tau))$ and is the action of $\overline{A}^{Pet}_{a,\tau}$ is still predicted by \eqref{petgen-nat}?

While the answers to the questions above are not known, \cite[Remark 5.4]{Pe07} gives a reason to be skeptical of applying \eqref{petgen-nat} to arbitrary functions. Namely, assume that $\alpha_i\in\D[\overline{A}^{Pet}_{a,\tau}]$ and $\overline{A}^{Pet}_{a,\tau}\alpha_i$ is given by \eqref{petgen-nat}. Then the functions $\tau\alpha_i+a$ are eigenfunctions of $\overline{A}^{Pet}_{a,\tau}$ with eigenvalue $-\tau$. However it is known \cite[Section 4.2]{Pe07} that $-\tau$ is not an eigenvalue of $\overline{A}^{Pet}_{a,\tau}$, leading to a contradiction. This means that, even if $\overline{A}^{Pet}_{a,\tau}$ is defined at natural coordinates $\alpha_i$, the result $\overline{A}^{Pet}_{a,\tau}\alpha_i$ cannot be predicted using the right-hand side of \eqref{petgen-nat}.

\subsection{$A_{z,z';\theta}$ in natural coordinates.} Our goal is to generalize \eqref{petgen-nat}, at least on the level of formal expressions. To do it we introduce the following notation. Let $\mathbb Z^*=\mathbb Z\backslash\{0\}$. Recall that we use $\{x_i\}_{i\in\mathbb Z^*}$ to denote the natural coordinates: $x_i=\alpha_i$ for $i>0$ and $x_i=\beta_{-i}$ for $i<0$. We define \emph{modified coordinates} $\tilde x_i$ by setting $\tilde x_i=\sgn_\theta(i)x_i$ where
$$
\sgn_\theta(i)=\begin{cases}1&\qquad i>0\\ -\theta&\qquad i<0\end{cases}
$$
Then $\frac{\partial}{\partial\tilde x_i}=\sgn_{\theta}(i)^{-1}\frac{\partial}{\partial x_i}$. With this notation we have
$$
q_k=\sum_{i\geq 1}\alpha_i^{k+1}+\sum_{j\geq1}(-\theta)^k\beta^{k+1}_j=\sum_{i\in\mathbb Z^*}\sgn_\theta(i)^{-1}\tilde x_i^{k+1}.
$$

Define the following formal differential operator:
\begin{multline}\label{conjA} 
A^{nat}_{z,z';\theta}=\sum_{i,j}(\sgn_{\theta}(i)\tilde x_i\1_{i=j}-\tilde x_i\tilde x_j)\frac{\partial^2}{\partial \tilde x_i\partial\tilde x_j}
+\sum_i\left(z+z'-\theta\sgn_\theta(i)^{-1}-\theta^{-1}zz'\tilde x_i \right)\frac{\partial}{\partial \tilde x_i}\\
+2\theta\sum_{i\neq j}\frac{x_j}{\tilde x_i-\tilde x_j}\frac{\partial}{\partial \tilde x_i},
\end{multline}
where in all sums $i,j\in\mathbb Z^*$. We claim that for $u\in\Lambda^\circ$
\begin{equation}
\label{forgen-nat}
A_{z,z';\theta}u=A^{nat}_{z,z';\theta} u
\end{equation}
under the restriction to $\Omega_0$. Here by $A^{nat}_{z,z';\theta} u$ we mean the result of the procedure where $\frac{\partial}{\partial \tilde x_i}$ are treated as differential operators on $\mathbb R[q_1, q_2, \dots, \tilde x_{-2}, \tilde x_{-1}, \tilde x_{1}, \tilde x_2, \dots]$ with action 
$$
\frac\partial{\partial \tilde x_i} q_k=(k+1)\sgn_\theta(i)^{-1}\tilde x_i^k, \qquad \frac\partial{\partial \tilde x_i} \tilde x_j=\1_{i=j},
$$
and then the resulting infinite sums are computed treating $\tilde x_i$ as formal variables satisfying $\sum_{i\in\mathbb Z^*} x_i=1$\footnote{To be precise, we treat the resulting formal objects as collections $(f_{n,m})_{n,m\geq 1}$ of rational functions $f_{n,m}\in \mathbb R(x_{-m},\dots, x_n)$ which are regular at $x_{n}=0$, $x_{-m}=0$ and such that $f_{n,m}(x_{-m}, \dots, x_{n-1}, 0)=f_{n-1, m}(x_{-m}, \dots, x_{n-1})$, $f_{n,m}(0, x_{-m+1}, \dots, x_n)=f_{n, m-1}(x_{-m+1}, \dots, x_n)$. Note that $q_k$ can be straightforwardly interpreted as such objects.}. Note that, similarly to \eqref{petgen-nat}, we do not claim anything about evaluation of the generator at functions outside of $\Lambda^\circ$, we will discus it in the next subsection. We also remark that $\Omega_{0}$ is a dense subset of $\Omega$, so a function $u\in\Lambda^\circ\subset C(\Omega)$ is uniquely determined by its restriction to points satisfying $\sum_{i\in\mathbb Z^*} x_i=1$. %

\begin{proof}[Proof of \eqref{forgen-nat}] Throughout the proof $i,j$ always denote indices from $\mathbb Z^*$. Note that for $u\in\Lambda^\circ$
$$
\frac\partial{\partial \tilde x_i} u=\sum_{k\geq 1} (k+1)\sgn_\theta(i)^{-1}\tilde x_i^k \frac{\partial}{\partial q_k} u,
$$
where the sum on the right-hand side has only finitely many non-zero terms. Similarly we have
$$
\frac{\partial^2}{\partial \tilde x_i\partial \tilde x_j} u=\1_{i=j}\sum_{k\geq 1} k(k+1)\sgn_\theta(i)^{-1}\tilde x_i^{k-1} \frac{\partial}{\partial q_k} u+\sum_{k,l\geq 1}(k+1)(l+1)\sgn_\theta(i)^{-1}\sgn_\theta(j)^{-1}\tilde x_i^k\tilde x_j^l \frac{\partial^2}{\partial q_k\partial q_l} u.
$$
This leads to 
$$
A^{nat}_{z,z';\theta}u=\sum_{k,l\geq 1} a_{k,l}(k+1)(l+1)\frac{\partial^2}{\partial q_k\partial q_l} u+ \sum_{k\geq 1} b_k (k+1)\frac{\partial}{\partial q_k} u,
$$
where
$$
a_{k,l}=\sum_{i,j}(\sgn_{\theta}(i)\tilde x_i\1_{i=j}-\tilde x_i\tilde x_j)\sgn_\theta(i)^{-1}\sgn_\theta(j)^{-1}\tilde x_i^k\tilde x_j^l
$$
\begin{multline*}
b_k=\sum_i (\sgn_{\theta}(i)\tilde x_i-\tilde x_i^2) k\sgn_\theta(i)^{-1}\tilde x_i^{k-1}+\sum_i\left(z+z'-\theta\sgn_\theta(i)^{-1}-\theta^{-1}zz'\tilde x_i \right)\sgn_\theta(i)^{-1}\tilde x_i^k\\
+2\theta\sum_{i\neq j}\frac{x_j}{\tilde x_i-\tilde x_j}\sgn_\theta(i)^{-1}\tilde x_i^k.
\end{multline*}
For $a_{k,l}$ we have
$$
a_{k,l}=\sum_{i}\sgn_\theta(i)^{-1}\tilde x_i^{k+l+1}-\sum_{i,j}\sgn_\theta(i)^{-1}\sgn_\theta(j)^{-1}\tilde x_i^{k+1}\tilde x_j^{l+1}=q_{k+l}-q_kq_l.
$$
To compute the coefficients $b_k$ note that 
$$
-\sum_ik\sgn_\theta(i)^{-1}\tilde x_i^{k+1}+\sum_i\left(z+z'-\theta^{-1}zz'\tilde x_i \right)\sgn_\theta(i)^{-1}\tilde x_i^k=-kq_k+(z+z')q_{k-1}-\theta^{-1}zz'q_k,
$$
where for $k=1$ we use the relation $\sum_i \sgn_\theta(i)^{-1} \tilde x_i=\sum_i x_i=1=q_0$. Then we get
$$
b_k=-(k+\theta^{-1}zz')q_k+(z+z')q_{k-1}+c_k,
$$
$$
c_k=\sum_i k\tilde x_i^{k}-\sum_i\theta\sgn_\theta(i)^{-2}\tilde x_i^k+2\theta\sum_{i\neq j}\frac{x_j}{\tilde x_i-\tilde x_j}\sgn_\theta(i)^{-1}\tilde x_i^k.
$$
Note that
\begin{multline*}
2\theta\sum_{i\neq j}\frac{x_j}{\tilde x_i-\tilde x_j}\sgn_\theta(i)^{-1}\tilde x_i^k=2\theta\sum_{ i< j}x_ix_j\frac{\tilde x_i^{k-1}-\tilde x_j^{k-1}}{\tilde x_i-\tilde x_j}=2\theta\sum_{i< j}\sum_{l=0}^{k-2}x_i\tilde x^{l}_ix_j\tilde x^{k-l-2}_j\\
=\theta\sum_{l=0}^{k-2}\sum_{i, j}x_i\tilde x^{l}_ix_j\tilde x^{k-l-2}_j-\theta(k-1)\sum_{i}x^2_i\tilde x^{k-2}_i.
\end{multline*}
Then
$$
c_k=\sum_ik(1-\theta\sgn_\theta(i)^{-2})\tilde x_i^{k}+\theta\sum_{l=0}^{k-2}\sum_{ i, j}x_i\tilde x^{l}_ix_j\tilde x^{k-l-2}_j=k(1-\theta)q_{k-1}+\theta\sum_{l=0}^{k-2}q_lq_{k-l-2},
$$
where we have used the relation $1-\theta\sgn_\theta(i)^{-2}=(1-\theta)\sgn_\theta(i)^{-1}$. Combining the expressions for $a_{k,l}, b_k, c_k$ above and comparing with $A_{z,z';\theta}$ gives \eqref{forgen-nat}.
\end{proof}

\subsection{Action of $\overline{A}_{z,z';\theta}$ on natural coordinates} The formal equation \eqref{forgen-nat} leads to the following question:

\begin{quest*} Let $\calP_{nat}=\mathbb R[\alpha_1,\alpha_2,\dots, \beta_1,\beta_2,\dots]$ and $\overline{A}_{z,z';\theta}$ be the closure of $A_{z,z';\theta}$ in $L^2(\Omega, \M)$. Is it true that $\calP_{nat}\subset \D[\overline{A}_{z,z';\theta}]$ and the restriction of $\overline{A}_{z,z';\theta}$ to $\calP_{nat}$ is given by $A^{nat}_{z,z';\theta}$?
\end{quest*}

We have no answer to the question above. However we can elaborate on various arguments towards the positive or negative answer. Recall that in the case $\theta\to 0$ corresponding to Petrov's diffusion we know that the answer is negative. However, for $\theta\neq 0$ the operator $A^{nat}_{z,z';\theta}$ introduces expressions of the form $\frac{1}{\alpha_i-\alpha_j}$, $\frac{1}{\beta_i-\beta_j}$, which might both change the situation and introduce new difficulties at the same time.

First note that it is not clear why expressions $\frac{1}{\alpha_i-\alpha_j}$, $\frac{1}{\beta_i-\beta_j}$ should be in $L^2(\Omega, \M)$ in the first place, so defining $A^{nat}_{z,z';\theta}$ as an operator in $L^2(\Omega, \M)$ is already nontrivial. In fact, these terms are the main reason why we treat \eqref{forgen-nat} only as an identity of formal expressions: one can see that the sums $\sum_{i,j} \frac{\tilde x_i\tilde x^k_j}{\tilde x_i-\tilde x_j}$ might fail to converge absolutely even when all $\tilde x_i$ are distinct, which leads to challenges in justifying the rearrangements of summands used in our proof of \eqref{forgen-nat}. 

On the other hand, because of the terms $(\tilde x_i-\tilde x_j)^{-1}$ we cannot immediately find simple eigenfunctions of $A^{nat}_{z,z';\theta}$. So the argument from Petrov's case based on the eigenstructure of $A_{z,z';\theta}$ seems not to be easily applicable in our case, giving some hope for the positive answer. Moreover, if one can show that $A^{nat}_{z,z';\theta}$ is a well-defined operator $\calP_{nat}\to L^2(\Omega, \M)$, a computation involving our approximations $\chi_C$ might lead to the positive answer. 

To support the previous paragraph we finish this section by computing the pointwise limit of $\overline{A}_{z,z';\theta}\chi_C$ on a certain subset of $\Omega$. Our answer turns out to be consistent with $A^{nat}_{z,z';\theta}\alpha_1$. However, note that this computation by itself is not enough to conclude that $\alpha_1\in\D[\overline{A}_{z,z';\theta}]$, since one also needs some $L^2$-integrable uniform bound to show convergence in $L^2(\Omega,\M)$. To keep this part brief and since this is not an actual proof, we only describe the key points of our computation leaving out majority of technical details.

~

\emph{Outline of the computation of $\lim_{C\to\infty}\overline{A}_{z,z';\theta}\chi_C$:} By an argument similar to Proposition \ref{prop111} one can show that the functions $\chi_C=C^{-1}\ln(1+\calQ[e^{Ct}])$ are in $\D[\overline{A}_{z,z';\theta}]$ and this extension is defined by the rules
$$
\frac{\partial}{\partial q_k}\chi_C=\frac{C^{k-1}}{k!(1+\calQ[e^{Ct}])}\qquad \frac{\partial}{\partial q_k}\frac{\partial}{\partial q_l}\chi_C=-\frac{C^{k+l-1}}{k!l!(1+\calQ[e^{Ct}])^2}.
$$
Then we have
\begin{multline}
\label{Agenequationchi}
A_{z,z';\theta}\chi_C=-\sum_{k,l\geq 1}(k+1)(l+1)(q_{k+l}-q_kq_l)\frac{C^{k+l-1}}{k!l!(1+\calQ[e^{Ct}])^2}\\
+\sum_{k\geq 1}(k+1)\left(((1-\theta)k+z+z')q_{k-1}-(k+\theta^{-1}zz')q_k\right)\frac{C^{k-1}}{k!(1+\calQ[e^{Ct}])}\\
+\theta\sum_{k,l\geq 0}(k+l+3)q_kq_l\frac{C^{k+l+1}}{(k+l+2)!(1+\calQ[e^{Ct}])}.
\end{multline}
Now we compute the limit of \eqref{Agenequationchi} as $C\to\infty$ at a point $\omega$ such that all $\alpha_i, -\theta\beta_j$ are different and $\sum_{i\geq 1}\alpha_i+\sum_{j\geq1}\beta_j=1$. We are going to do it term by term, actively using the properties of the functions $\calQ[\phi]$ proved in Section \ref{basicQ-sect}, in particular the asymptotic behavior of $\calQ[t^ke^{Ct}]$ from Lemma \ref{exasympt}.

\paragraph{\bfseries Term 1:}  Note that
$$
-\sum_{k,l\geq 1}(k+1)(l+1)q_{k+l}\frac{C^{k+l-1}}{k!l!(1+\calQ[e^{Ct}])^2}=-\frac{\calQ[\sum_{k,l\geq 1}\frac{(k+1)(l+1)(Ct)^{k+l}}{k!l!}]}{C(1+\calQ[e^{Ct}])^2}.
$$
We have
\begin{equation}\label{tmpcollapseeq}
\sum_{k\geq 1}\frac{(k+1)(Ct)^{k}}{k!}=(1+Ct)e^{Ct}+ \text{\ polynomial\ in\ } C, t.
\end{equation}
Then, in view of Lemma \ref{exasympt} as $C\to\infty$
$$
-\frac{\calQ[\sum_{k,l\geq 1}\frac{(k+1)(l+1)(Ct)^{k+l}}{k!l!}]}{C(1+\calQ[e^{Ct}])^2}=-\frac{\alpha_1(1+C\alpha_1)^2e^{2C\alpha_1}+o(e^{2C\alpha_1})}{C(\alpha_1e^{C\alpha_1}+o(e^{C\alpha_1}))^2}=-2-C\alpha_1+o(1)
$$

\paragraph{\bfseries Term 2:} We have
$$
\sum_{k,l\geq 1}(k+1)(l+1)q_kq_l\frac{C^{k+l-1}}{k!l!(1+\calQ[e^{Ct}])^2}=C^{-1}\left(\frac{\calQ[\sum_{k\geq 1}\frac{(k+1)(Ct)^{k}}{k!}]}{1+\calQ[e^{Ct}]}\right)^2.
$$
By \eqref{tmpcollapseeq} and Lemma \ref{exasympt} this is equal to
$$
C^{-1}\left(\frac{\alpha_1(1+C\alpha_1)e^{C\alpha_1}+o(e^{C\alpha_1})}{\alpha_1e^{C\alpha_1}+o(e^{C\alpha_1})}\right)^2=C^{-1}(1+C\alpha_1)^2+o(1)=2\alpha_1+C\alpha_1^2+o(1).
$$

\paragraph{\bfseries Term 3:} 
\begin{multline*}
\sum_{k\geq 1}(k+1)\left(((1-\theta)k+z+z')q_{k-1}-(k+\theta^{-1}zz')q_k\right)\frac{C^{k-1}}{k!(1+\calQ[e^{Ct}])}\\
=\frac{\calQ[\sum_{k\geq 1}(k+1)\frac{((1-\theta)k+z+z')t^{k-1}C^k-(k+\theta^{-1}zz')(tC)^k}{k!}]}{C(1+\calQ[e^{Ct}])}.
\end{multline*}
By a direct manipulation with power series one can verify that
\begin{multline*}
\sum_{k\geq 1}(k+1)\frac{((1-\theta)k+z+z')t^{k-1}C^k-(k+\theta^{-1}zz')(tC)^k}{k!}\\
=(1-\theta)C(2+Ct)e^{Ct}+(z+z')(t^{-1}+C)(e^{Ct}-1)-(2+Ct)tCe^{Ct}-\theta^{-1}zz'(1+Ct)e^{Ct}+ \text{\ polynomial\ in\ } C,t.
\end{multline*}
By Lemma \ref{exasympt} this term is equal to
$$
\frac{\alpha_1\left((1-\theta)C(2+C\alpha_1)e^{C\alpha_1}+(z+z')(\alpha_1^{-1}+C)e^{C\alpha_1}-(2+C\alpha_1)\alpha_1Ce^{C\alpha_1}-\theta^{-1}zz'(1+C\alpha_1)e^{C\alpha_1}\right)+o(e^{C\alpha_1})}{C(\alpha_1e^{C\alpha_1}+o(e^{C\alpha_1}))}
$$
Simplifying, we get 
$$.
(1-\theta)(2+C\alpha_1)+z+z'-(2+C\alpha_1)\alpha_1-\theta^{-1}zz'\alpha_1+o(1)
$$

\paragraph{\bfseries Term 4:} We come to the most challenging term
$$
\theta\sum_{k,l\geq 0}(k+l+3)q_kq_l\frac{C^{k+l+1}}{(k+l+2)!(1+\calQ[e^{Ct}])}.
$$
Replacing $q_k$ by $\sum_{i\in\mathbb Z^*}x_i\tilde{x_i}^k$ we get 
$$
\theta\sum_{k,l\geq 0}\sum_{i,j\in\mathbb Z^*}(k+l+3)x_i\tilde x_i^kx_j\tilde x_j^l\frac{C^{k+l+1}}{(k+l+2)!(1+\calQ[e^{Ct}])}.
$$
To study this expression we swap the order of summation and split the term into two parts
$$
P_1=\theta\sum_{i\in\mathbb Z^*}\sum_{k,l\geq 0}(k+l+3)x_i^2\tilde x_i^{k+l}\frac{C^{k+l+1}}{(k+l+2)!(1+\calQ[e^{Ct}])}
$$
$$
P_2=\theta\sum_{i\neq j\in\mathbb Z^*}\sum_{k,l\geq 0}(k+l+3)x_i\tilde x_i^kx_j\tilde x_j^l\frac{C^{k+l+1}}{(k+l+2)!(1+\calQ[e^{Ct}])}.
$$
For $P_1$ by setting $n=k+l$ we get
$$
P_1=\theta\sum_{i\in\mathbb Z^*}\sum_{n\geq 0}(n+3)(n+1)x_i^2\tilde x_i^{n}\frac{C^{n+1}}{(n+2)!(1+\calQ[e^{Ct}])}.
$$
Note that
$$
\sum_{n\geq 0}(n+3)(n+1)x_i^2\tilde x_i^{n}\frac{C^{n+1}}{(n+2)!}=-\sgn_\theta(i)^{-2}C^{-1}(e^{C\tilde x_i}-1-C\tilde x_i)+x_i\sgn_\theta(i)^{-1}(e^{C\tilde x_i}-1)+x_i^2Ce^{C\tilde x_i}
$$
Then
$$
P_1=\theta\sum_{i\in\mathbb Z^*}\frac{-\sgn_\theta(i)^{-2}C^{-1}(e^{C\tilde x_i}-1-C\tilde x_i)+x_i\sgn_\theta(i)^{-1}(e^{C\tilde x_i}-1)+x_i^2Ce^{C\tilde x_i}}{1+\calQ[e^{Ct}]}
$$
and as $C\to\infty$
$$
P_1=\theta\sum_{i\in\mathbb Z^*}\frac{-\sgn_\theta(i)^{-2}C^{-1}(e^{C\tilde x_i}-1-C\tilde x_i)+x_i\sgn_\theta(i)^{-1}(e^{C\tilde x_i}-1)+x_i^2Ce^{C\tilde x_i}}{(\alpha_1+o(1))e^{C\alpha_1}}.
$$
Since $\alpha_1>\tilde x_i$ for any $i\neq 1$, all the summands above except $i=1$ decay exponentially and we have
$$
P_1=\theta\frac{-C^{-1}+\alpha_1+\alpha_1^2C}{\alpha_1}+o(1)=\theta(1+\alpha_1C)+o(1).
$$
For the second part note that for each fixed $i\neq j$ we can use the relation
$$
\sum_{k+l=n}\tilde x_i^k\tilde x_j^l=\frac{\tilde x_i^{n+1}-\tilde x_j^{n+1}}{\tilde x_i-\tilde x_j}
$$
to get:
$$
P_2=\theta\sum_{i\neq j\in\mathbb Z^*}\sum_{n\geq 0}(n+3)x_ix_j\frac{\tilde x_i^{n+1}-\tilde x_j^{n+1}}{\tilde x_i-\tilde x_j}\frac{C^{n+1}}{(n+2)!(1+\calQ[e^{Ct}])}.
$$
Note that 
$$
\sum_{n\geq 0}(n+3)\tilde x_i^{n+1}\frac{C^{n+1}}{(n+2)!}=\tilde x_i^{-1}C^{-1}(e^{C\tilde x_i}-1-C\tilde x_i)+e^{C\tilde x_i}-1  
$$
So
$$
P_2=2\theta\sum_{i<j}x_ix_j\frac{\tilde x_i^{-1}C^{-1}(e^{C\tilde x_i}-1-C\tilde x_i)+(e^{C\tilde x_i}-1)-\tilde x_j^{-1}C^{-1}(e^{C\tilde x_j}-1-C\tilde x_j)-(e^{C\tilde x_j}-1)}{(\tilde x_i-\tilde x_j)(\alpha_1+o(1))e^{C\alpha_1}}
$$
Again, since $\alpha_1>\tilde x_i$ for all $i\neq 1$, all the summands decay exponentially with exception of the case when either $i=1$ or $j=1$. So
$$
P_2=2\theta\sum_{i\neq 1}\alpha_1x_i\frac{\alpha_1^{-1}C^{-1}+1}{(\alpha_1-\tilde x_i)(\alpha_1+o(1))} +o(1)=2\theta\sum_{i\neq 1}\frac{x_i}{\alpha_1-\tilde x_i} +o(1).
$$
So, in total the contribution of the last term is
$$
\theta(1+\alpha_1C)+2\theta\sum_{i\neq 1}\frac{x_i}{\alpha_1-\tilde x_i} +o(1).
$$
Now, combining all the terms computed above, we get that $A_{z,z';\theta}\chi_C(\omega)$ as $C\to\infty$ is given by
$$
-2-C\alpha_1+2\alpha_1+C\alpha_1^2+(1-\theta)(2+C\alpha_1)+z+z'-(2+C\alpha_1)\alpha_1-\theta^{-1}zz'\alpha_1+\theta(1+\alpha_1C)+2\theta\sum_{i\neq 1}\frac{x_i}{\alpha_1-\tilde x_i} +o(1)
$$
Note that all the terms with $C$ cancel out and the limit as $C\to\infty$ is
$$
-\theta+z+z'-\theta^{-1}zz'\alpha_1+2\theta\sum_{i\neq 1}\frac{x_i}{\alpha_1-\tilde x_i}.
$$
This is exactly $A_{z,z';\theta}^{nat}\alpha_1$.


\begin{thebibliography}{999999999}

\bibitem[Bor98]{Bor98} A. Borodin, Point Processes and the Infinite Symmetric Group. Part II: Higher Correlation Functions. 1998; {\tt arXiv:math/9804087}.

\bibitem[BO98]{BO98} A. Borodin, G. Olshanski, Point processes and the infinite symmetric group. \emph{Math.Res.Lett.} 5, pp. 799-816, 1998; {\tt arXiv:math/9810015}.

\bibitem[BO99]{BO99} A. Borodin, G. Olshanski, z-measures on partitions, Robinson-Schensted-Knuth correspondence, and $\beta$=2 random matrix ensembles. \emph{Random matrix models and their applications}, pp. 71-94, Math. Sci. Res. Inst. Publ., 40, Cambridge Univ. Press, Cambridge, 2001; {\tt arXiv:math/9905189}.

\bibitem[BO02]{BO02} A. Borodin, G. Olshanski, Z-measures on partitions and their scaling limits. \emph{European J. Combin.} 26(6), pp 795-834, 2005; {\tt arXiv:math-ph/0210048}.

\bibitem[BO07]{BO07} A. Borodin, G. Olshanski, Infinite-dimensional diffusions as limits of random walks on partitions. \emph{Prob. Theor. Rel. Fields 144}, no. 1, pp 281-318, 2009; {\tt arXiv:0706.1034}.

\bibitem[BOS05]{BOS05} A. Borodin, G. Olshanski, E. Strahov, Giambelli compatible point processes. \emph{Adv. in Appl. Math.} 37(2), pp 209-248, 2006; {\tt arXiv:math-ph/0505021}.

\bibitem[CDERS16]{CDERS16} C. Costantini, P. De Blasi, S. N. Ethier, M. Ruggiero, and D. Spano, Wright–Fisher construction of the two-parameter Poisson–Dirichlet diffusion. \emph{Ann. Appl. Probab.}, 27(3), pp 1923-1950, 2017; {\tt arXiv:1601.06064}. 

\bibitem[Eth92]{Eth92} S. N. Ethier, Eigenstructure of the Infinitely-Many-Neutral-Alleles Diffusion Model. \emph{Journal of Applied Probability}, 29(3), pp. 487-498, 1992.

\bibitem[Eth14]{Eth14} S. N. Ethier, A property of Petrov's diffusion. \emph{Electronic Communications in Probability}, 19(65), pp 1-4, 2014; {\tt arXiv:1406.5198}.

\bibitem[EK81]{EK81} S. N. Ethier, T. G. Kurtz. The infinitely-many-neutral-allels diffusion model, \emph{Advances in Applied Probability}, 13(3), pp 429-452, 1981.

\bibitem[FOT10]{FOT10} M. Fukushima, Y. Oshima, M. Takeda, \emph{Dirichlet Forms and Symmetric Markov Processes.} Berlin, New York: De Gruyter, 2010.

\bibitem[FPRW21]{FPRW21} N. Forman, S. Pal, D. Rizzolo, M. Winkel, Ranked masses in two-parameter Fleming-Viot diffusions. \emph{Trans. Amer. Math. Soc.} 376, pp 1089-1111, 2023; {\tt arXiv:2101.09307}.

\bibitem[FS09]{FS09} S. Feng, W. Sun, Some Diffusion Processes Associated With Two Parameter Poisson-Dirichlet Distribution and Dirichlet Process.  \emph{Probab. Theory Relat. Fields} 148, pp 501-525, 2010; {\tt arXiv:0903.0623}.

\bibitem[FSWX11]{FSWX11} S. Feng, W. Sun, F.-Y. Wang and F. Xu, Functional inequalities for the two-parameter extension of the infinitely-many-neutral-alleles diffusion, \emph{J. Funct. Anal.} 260, pp 399-413, 2011.

\bibitem[K18]{K18} S. Korotkikh, Transition functions of diffusion processes with the Jack parameter on the Thoma simplex. \emph{Functional Analysis and Its Applications}, 54(2), pp 118-134, 2020; {\tt arXiv:1806.07454}.

\bibitem[KOV93]{KOV93}  S. Kerov, G. Olshanski, A. Vershik, Harmonic analysis on the infinite symmetric group. A deformation of the regular representation. \emph{Comptes Rendus de l’Acad\'emie des Sciences, S\'erie I}, 316, pp 773-778, 1993.

\bibitem[KOV03]{KOV03} S. Kerov, G. Olshanski, A. Vershik, Harmonic analysis on the infinite symmetric group. \emph{Invent. Math.} 158(3), pp 551-642, 2004; {\tt arXiv:math/0312270}.

\bibitem[Lig10]{Lig10} T. M. Liggett, \emph{Continuous Time Markov Processes: An Introduction}; Graduate Studies in Mathematics 113, American Mathematical Soc.; 2010.

\bibitem[Ok00]{Ok00} A. Okounkov, SL(2) and z-measures. \emph{Random matrix models and their applications.} pp. 407-420. Math. Sci. Res. Inst. Publ., 40, Cambridge Univ. Press, Cambridge, 2001; {\tt arXiv:math/0002135}.

\bibitem[Olsh98]{Olsh98} G. Olshanski, Point processes and the infinite symmetric group. Part I: The general formalism and the density function. \emph{The orbit method in geometry and physics}: in honor of A. A. Kirillov (C. Duval, L. Guieu, V. Ovsienko, eds), Progress in Math. 213. Birkhauser, pp. 349-393, 2003; {\tt arXiv:math/9804086}.

\bibitem[Olsh09]{Olsh09} G. Olshanski, Anisotropic Young diagrams and infinite-dimensional diffusion processes with the Jack parameter. \emph{International Mathematics Research Notices} no. 6, pp 1102-1166, 2010; {\tt arXiv:0902.3395}.

\bibitem[Olsh18]{Olsh18} G. Olshanski, The topological support of the z-measures on the Thoma simplex. \emph{Functional Analysis and its Applications}, 52(4), 308-310, 2018; {\tt arXiv:1809.07125}.

\bibitem[Pe07]{Pe07} L. Petrov, A Two-Parameter Family of Infinite-Dimensional Diffusions in the Kingman Simplex. \emph{Functional Analysis and Its Applications}, 43(4), pp 279-296, 2009; {\tt arXiv:0708.1930}.

\bibitem[PY97]{PY97} J. Pitman, M. Yor, The two-parameter Poisson-Dirichlet distribution derived from a stable subordinator. \emph{Ann. Probab.}, 25(2), pp 855-900, 1997.

\bibitem[RW09]{RW09} M. Ruggiero and S. G. Walker, Countable representation for infinite dimensional diffusions derived from the two-parameter Poisson–Dirichlet process. \emph{Electron. Commun. Probab.}, 14, pp 501-517, 2009.

\bibitem[Sch91]{Sch91} B. Schmuland, A Result on the Infinitely Many Neutral Alleles Diffusion Model. \emph{Journal of Applied Probability} 28(2), pp 253-267, 1991.

\end{thebibliography}
\end{document}